\documentclass[final,onefignum,onetabnum]{siamart190516}

\usepackage{booktabs}       

\usepackage{amsmath}
\usepackage{amssymb}
\usepackage{graphicx}
\usepackage{caption}
\usepackage{subcaption}
\usepackage{color}

\usepackage{footnote}
\usepackage{siunitx}

\usepackage{pgfplotstable}
\pgfplotsset{compat=1.9}

\usepackage[nocompress]{cite} 
\DeclareMathOperator{\trace}{trace} 

\usepackage[letterpaper,centering]{geometry}
\usepackage[textsize=tiny]{todonotes}
\usepackage{lipsum}
\usepackage{amsfonts}
\usepackage{graphicx}
\usepackage{epstopdf}
\usepackage{algorithmic}
\ifpdf
  \DeclareGraphicsExtensions{.eps,.pdf,.png,.jpg}
\else
  \DeclareGraphicsExtensions{.eps}
\fi

\newsiamremark{remark}{Remark}
\newsiamremark{hypothesis}{Hypothesis}
\crefname{hypothesis}{Hypothesis}{Hypotheses}
\newsiamthm{claim}{Claim}

\headers{Multifidelity Dimension Reduction via Active Subspaces}{R.\ Lam, O.\ Zahm, Y.\ Marzouk, and K.\ Willcox}

\title{Multifidelity Dimension Reduction\\ via Active Subspaces\thanks{Submitted to the editors on 09/14/2018.
\funding{This work was funded in part by the AFOSR MURI on multi-information sources of multi-physics systems under Award Number FA9550-15-1-0038, program manager Dr.\ Fariba Fahroo.}}}

\author{Remi R.\ Lam\thanks{Department of Aeronautics and Astronautics, Massachusetts Institute of Technology, Cambridge, MA 02139
  	(\email{rlam@mit.edu}, \email{ymarz@mit.edu}).}
\and Olivier Zahm\thanks{Universit\'e Grenoble Alpes, Inria, CNRS, Grenoble INP, LJK, 38000 Grenoble, France 
  	(\email{olivier.zahm@inria.fr}).}
\and Youssef M.\ Marzouk\footnotemark[2]
\and Karen E.\ Willcox\thanks{Oden Institute for Computational Engineering and Sciences, UT Austin, Austin, TX 78712 (\email{kwillcox@oden.utexas.edu})}}

\usepackage{amsopn}

\ifpdf
\hypersetup{ 
  pdftitle={Multifidelity Dimension Reduction via Active Subspaces},
  pdfauthor={Remi R.\ Lam, Olivier Zahm, Youssef M.\ Marzouk and Karen E.\ Willcox}
}
\fi

\begin{document}

\maketitle

\newcommand{\design }{\textcolor{black}{\boldsymbol{x}}}
\newcommand{\iter}{\textcolor{black}{n}}
\newcommand{\stage}{\textcolor{black}{k}}
\newcommand{\obj}{\textcolor{black}{f}}
\newcommand{\obs}{\textcolor{black}{y}}
\newcommand{\OBS}{\textcolor{black}{Y}}
\newcommand{\surr}{\textcolor{black}{\overline{\mu}}}
\newcommand{\svar}{\textcolor{black}{\overline{\sigma}}}
\newcommand{\trainset}{\textcolor{black}{\mathcal{S}}}
\newcommand{\hyp}{\textcolor{black}{\boldsymbol{\theta}}}
\newcommand{\dimension}{\textcolor{black}{d}}
\newcommand{\maxiter}{\textcolor{black}{N}}
\newcommand{\normal}{\textcolor{black}{\mathcal{N}}}
\newcommand{\expect}{\textcolor{black}{\mathbb{E}}}
\newcommand{\utility}{\textcolor{black}{U}}
\newcommand{\utilityfunc}{\textcolor{black}{\utility_{\iter}}}
\newcommand{\Pol}{\textcolor{black}{\boldsymbol{\pi}}}
\newcommand{\pol}{\textcolor{black}{\pi}}

\newcommand{\feasibletarget}{\obj_{best}^{\trainset_{\textcolor{black}{k}}}}
\newcommand{\cindex}{i}
\newcommand{\nc}{I}
\newcommand{\disc}{\gamma}
\newcommand{\constraint}{g}
\newcommand{\argminx}{\operatornamewithlimits{\text{argmin}}\limits_{\design\in\mathcal{X}} }
\newcommand{\disturb}{\boldsymbol{w}}
\newcommand{\rdesign}{\boldsymbol{x}^{*}}
\newcommand{\fpenalty}{\Psi}

\newcommand{\grad}{\nabla f}
\newcommand{\approxfunc}{g}
\newcommand{\gradapprox}{\nabla g}
\newcommand{\density}{\rho}
\newcommand{\nfid}{k}
\newcommand{\inputspace}{\mathcal{X}}
\newcommand{\aspace}{\inputspace_{r}}
\newcommand{\proba}{\mathbb{P}}

\newcommand{\matrixH}{H}
\newcommand{\Hhat}{\widehat{H}}
\newcommand{\matrixG}{G}
\newcommand{\intdim}{\delta_H}
\newcommand{\varstat}{v}

\newcommand{\normH}{\Vert\matrixH\Vert}
\newcommand{\normE}{\expect[\Vert\grad(X)\Vert^{2}]}
\newcommand{\lambdav}{\boldsymbol{\lambda}}
\newcommand{\unispace}{\mathcal{X}}
\newcommand{\mleq}{\preccurlyeq}
\newcommand{\mgeq}{\succcurlyeq}
\newcommand{\sampvec}{\boldsymbol{m}}
\newcommand{\randd}{X}

\newcommand{\HMF}{\widehat{\matrixH}_{MF}}
\newcommand{\HSF}{\widehat{\matrixH}_{SF}}

\newcommand{\vSF}{v_{\!S\!F}}
\newcommand{\LSF}{L_{\!S\!F}}
\newcommand{\vMF}{v_{\!M\!F}}
\newcommand{\LMF}{L_{\!M\!F}}

\newcommand{\real}{\mathbb{R}}

\newcommand{\rkhs}{\mathcal{H}_{\kappa}}
\newcommand{\R}{\mathbb{R}}
\newcommand{\ridge}{\mathcal{R}}
\newcommand{\gradridge}{\nabla\mathcal{R}}
\newtheorem{property}[theorem]{Property}

\begin{abstract}
	We propose a multifidelity dimension reduction method to identify a low-dimensional structure present in many engineering models.
	The structure of interest arises when functions vary primarily on a low-dimensional subspace of the high-dimensional input space, while varying little along the complementary directions.
	Our approach builds on the gradient-based methodology of active subspaces, and exploits models of different fidelities to reduce the cost of performing dimension reduction
	through the computation of the active subspace matrix.
	We provide a non-asymptotic analysis of the number of gradient evaluations sufficient to achieve a prescribed error 
	in the active subspace matrix,
	both in expectation and with high probability.
	We show that the sample complexity depends on a notion of intrinsic dimension of the problem, which can be much smaller than the dimension of the input space.
	We illustrate the benefits of such a multifidelity dimension reduction approach using numerical experiments with input spaces of up to two thousand dimensions.
\end{abstract}


\begin{AMS}{
  15A18, 15A60, 41A30, 41A63, 65D15, 65N30}
\end{AMS}

\section{Introduction}
\label{sec:introduction}
Engineering models are typically parameterized by a large number of input variables, and can also be expensive to evaluate. 
Yet these models are often embedded in problems of global optimization or uncertainty quantification, whose computational cost and complexity increase dramatically with the number of model inputs.
One strategy to circumvent this \textit{curse of dimensionality} is to exploit, when present, some notion of low-dimensional structure and to perform \textit{dimension reduction}. 
Doing so can significantly reduce the complexity of the problem at hand.
In this paper, we consider the problem of identifying the low-dimensional structure that arises when an output of a model varies primarily on a low-dimensional subspace of the input space, while varying little along the complementary directions.
This structure is commonly found in engineering problems and can be identified using the \textit{active subspace} method \cite{russi2010uncertainty,constantine2015active}, among other methods.
The active subspace method
relies on the computation of a second moment matrix, a step that can be costly as it often involves many evaluations of the gradient of the model. 
In this work, we consider the common engineering setting where cheap low-fidelity approximations of an expensive high-fidelity model, and its gradients, are available. 
We propose a \textit{multifidelity} gradient-based algorithm to reduce the cost of performing dimension reduction via active subspaces. In particular, we present a multifidelity estimator of the second moment matrix used by the active subspace method and show, theoretically and empirically, that fewer evaluations of the expensive gradient are sufficient to perform dimension reduction.

Several approaches have been devised to identify low-dimensional structure in the input space of a function.
These methods include global sensitivity analysis \cite{saltelli2008global}, sliced inverse regression \cite{li1991sliced}, basis adaptation \cite{tipireddy2014basis}, and low-rank matrix recovery \cite{tyagi2014learning}.
Recent work has also explored combining dimension reduction in both the input and the state space of the associated model \cite{lieberman2010parameter,himpe2015data,salmoiraghi2016isogeometric,cui2016scalable,tezzele2018combined,ballarin2019pod}.
Such methods typically require a large number of (potentially expensive) function evaluations.
When derivative information is available (e.g., via adjoint methods or automatic differentiation), 
gradient-based methods have also been proposed to detect the low-dimensional structure of a smooth function, with higher sample efficiency \cite{samarov1993exploring}. 
One way to leverage derivative information is to examine the spectral properties of the second moment matrix of the gradient of the function. 
The dominant eigenspace of that matrix contains the directions along which the function, loosely speaking, varies the most.
This dominant eigenspace is called the \textit{active subspace} \cite{russi2010uncertainty,constantine2014active,constantine2015active}. 
More precisely, in \cite{zahm2018gradient}, the second moment matrix is used to construct an upper bound for the function approximation error induced by dimension reduction.
The active subspace's dimension is then chosen in order to satisfy a user-defined tolerance, allowing a rigorous control of the approximation error.
Gradient-based methods have been successfully used to detect and exploit low-dimensional structure in engineering models \cite{lukaczyk2014active,jefferson2017exploring,constantine2017time,ji2018shared} as well as in Bayesian inverse problems \cite{cui2014likelihood,constantine2016accelerating,zahm2018certified}. 
The efficiency of these gradient-based methods depends upon the computation of the second moment matrix of the gradient.
This can be an expensive step as it involves computing an integral, over the high-dimensional input space, of the gradient of an expensive function.
Reducing the cost of the dimension reduction step is particularly important as it allows more computational resources to be allocated to the original task of interest (e.g., optimization or uncertainty quantification).

To reduce this computational cost, one strategy consists of replacing the expensive gradient with a cheap-to-evaluate approximation or surrogate. 
Surrogates with lower evaluation cost are widely available in engineering problems:
they include models defined by numerically solving equations on coarser meshes, using simplified governing equations, imposing looser convergence criteria, or employing reduced-order models.
In order to control the error induced by the use of a surrogate, \textit{multifidelity methods} aim at combining cheap approximations with expensive but accurate information in an optimal way (see \cite{PWG17MultiSurvey} for a survey).
The goal of such approaches is to shift most of the work to the cheaper model, while querying the expensive model often enough to guarantee convergence to the desired quantity (in this case, the second moment matrix of the gradient).
For instance, multigrid methods use a hierarchy of cheaper and coarser discretizations to solve systems of partial differential equations more efficiently \cite{brandt1977multi,briggs2000multigridSIAM,hackbusch2013multi}.
In multilevel Monte Carlo, expected quantities and rare event probabilities are computed by distributing the computational work among several levels of approximation with known error rate and cost \cite{giles2008multilevel,teckentrup2015multilevel,ullmann2015multilevel,kuo2017multilevel,beskos2018smc}. 
When no such information about error rates is available, or when there is no hierarchy among models, multifidelity techniques have been employed to accelerate Monte Carlo estimates \cite{peherstorfer2016optimal} by solving an optimal resource allocation problem among a collection of models with varying fidelity.
Multifidelity techniques have also been devised to accelerate optimization \cite{alexandrov2001approximation,forrester2007multi,march2012provably,swersky2013multi,lam2015multifidelity,poloczek2017multi,kandasamy2017multi}, global sensitivity analysis \cite{qian2018multifidelity}, or importance sampling and rare event estimation \cite{li2010evaluation,li2011efficient,peherstorfer2017combining,peherstorfer2016multifidelity}.
While most multifidelity techniques have focused on estimating the expectations of scalar quantities, high-dimensional objects such as the second moment matrix in the active subspace method---effectively, the expectation of a matrix-valued function---have received less attention.
Because high-dimensional objects are typically more challenging to approximate, developing and analyzing multifidelity algorithms for their estimation could lead to significant computational savings. 

In this paper, we use multifidelity techniques to reduce the computational cost of performing dimension reduction. 
We build on the gradient-based active subspace method, proposing a multifidelity estimator for the second moment matrix that uses the low-fidelity model as a control variate for the outputs of the high-fidelity model---thus providing variance reduction and reducing computational costs. 
We establish non-asymptotic error bounds for this estimator, both in expectation and in high probability. We show that the sample complexity depends on the intrinsic dimension of the second moment matrix, a quantity that can be much smaller than the dimension of the input space when the function of interest varies mostly along a few directions. 
Finally, we demonstrate the performance of our proposed multifidelity dimension reduction technique on several analytical and engineering examples.

The paper is organized as follows. In Section \ref{sec:active_subspace}, we give a brief review of the active subspace methodology. 
Then, we formalize the proposed active subspace multifidelity algorithm in Section \ref{sec:multifidelity_active_subspace}.
Error bounds for the single-fidelity and multifidelity active subspace algorithms are provided in Section \ref{sec:theoretical_analysis}.
We illustrate the benefits of our approach with numerical examples in Section \ref{sec:numerical_results} before summarizing our findings in Section \ref{sec:conclusions}.

\section{Active subspace}
\label{sec:active_subspace}
We consider a scalar-valued function $\obj:\inputspace\to\R$ where the input space $\inputspace$ is a subset of $\R^{\dimension}$. 
We refer to the dimension $\dimension\in\mathbb{N}$ as the \textit{ambient dimension}.
The active subspace method \cite{constantine2014active,constantine2015active} aims to compute a low-dimensional subspace of $\inputspace$ in which most of the variations of $\obj$ are concentrated.
The active subspace method assumes that $\obj$ is differentiable and that each component of $\grad$ is square integrable on the space $\inputspace$, weighted by a user-defined probability density $\rho:\inputspace\to\R^{+}$. 
This guarantees the well posedness of the second moment matrix
\begin{align*}\label{eq:ASmatrixH}
  \matrixH = \expect\left[ \grad(X)\grad(X)^{T}\right],
\end{align*}
where $X\sim\rho$ is a random variable taking values in $\inputspace$ and $\expect[\,\cdot\,]$ denotes the expectation.
We refer to $\matrixH$ as the active subspace matrix (AS matrix).
The eigendecomposition of $\matrixH$ yields information about the directions along which $\obj$ varies.
Specifically, for any unit norm vector $u\in\R^{\dimension}$, the quantity
$u^{T}\matrixH u = \expect[(\grad(X)^{T}u)^2]$ corresponds to the $L^{2}$ norm of the gradient $\grad$ projected on $\text{span}\{u\}$. 
Thus, the largest eigenvector of $\matrixH$, which is a maximizer of $u^{T}\matrixH u$ over unit norm vectors $u\in\R^{\dimension}$, is aligned with the direction in which $\obj$ has largest (in squared magnitude) average derivative. 

Another important property is that, under some mild assumptions on the probability density $\rho$, the AS matrix allows us to control the mean square error between $\obj(X)$ and a ridge approximation of the form of $h(U_{r}^{T}X)$, 
where $U_{r}\in\mathbb{R}^{\dimension \times r}$ is a matrix with $r\leq\dimension$ orthonormal columns. 
In particular, if $h$ is defined to be the conditional expectation
$h(U_{r}^{T}X) = \expect[f(X) \vert U_{r}^{T}X]$, 
$\inputspace=\mathbb{R}^{\dimension}$, and $\rho$ is the density of the standard normal distribution on $\inputspace$, then Proposition 2.5 in \cite{zahm2018gradient} (with $P_{r}=U_{r}U_{r}^{T}$) guarantees that
\begin{equation}\label{eq:AScontrol}
 \expect [ ( \obj(X) - h(U_{r}^{T}X) )^{2} ] \leq \trace(H) - \trace(U_{r}^{T} \matrixH U_{r}) ,
\end{equation}
holds for any $U_{r}$ such that $U_{r}^{T}U_{r}=I_{r}$.
This result relies on Poincar\'e-type inequalities and can be extended to more general densities $\rho$ (see Corollary 2 in \cite{zahm2018certified}). 
In order to obtain a good approximation of $\obj$ in the $L^{2}$ sense, we can choose $U_{r}$ as a matrix which minimizes the right-hand side of \eqref{eq:AScontrol}. 
This is equivalent to the problem
\begin{equation}\label{eq:ASmaxtrace}
 \max_{\substack{U_{r}\in\mathbb{R}^{\dimension\times r} \\ \text{s.t. } U_{r}^{T} U_{r} = I_{r}}} 
 \trace(U_{r}^{T} \matrixH U_{r}).
\end{equation}
Any matrix $U_{r}$ whose columns span the $r$-dimensional dominant eigenspace of $\matrixH$ is a solution.
The corresponding subspace is called the \textit{active subspace}.

In practice, there is no closed-form expression for the AS matrix and $\matrixH$ must be approximated numerically. 
The following Monte Carlo estimator requires evaluating $\grad$ at $m_{1}$ realizations of the  input parameters, drawn independently from $\rho$. 
We refer to this estimator as a single-fidelity estimator (SF estimator).

\begin{definition}[Single-fidelity estimator]\label{def:SF_definition}
  Let $m_{1}\geq1$ be the number of gradient evaluations.
  We define the SF estimator of $\matrixH$ to be
  \begin{align*}
    \HSF =& \frac{1}{m_{1}}\sum_{i=1}^{m_{1}}\grad(\randd_{i})\grad(\randd_{i})^{T},
  \end{align*}
  where $\randd_{1},\ldots,\randd_{m_{1}}$ are independent copies of $X\sim\rho$.
\end{definition}

Computing an estimate of $\matrixH$ with a satisfactory error can require
a large number $m_{1}$ of gradient evaluations.
In the following section, we propose a new multifidelity algorithm that leverages a cheap-to-evaluate approximation of $\grad$ to reduce the cost of estimating $\matrixH$.

\section{Multifidelity dimension reduction}
\label{sec:multifidelity_active_subspace}
In this section, we describe a multifidelity approach for estimating the AS matrix $\matrixH$ (Sec.~\ref{sub:multifidelity_active_subspace_estimator}). We also characterize the impact of using such an approximation of $\matrixH$ on the quality of the dimension reduction (Sec.~\ref{sub:relationship_to_function_approximation}). 

\subsection{Multifidelity active subspace estimator} 
\label{sub:multifidelity_active_subspace_estimator}

Suppose we are given a function $\approxfunc:\inputspace\to\R$ that is a cheap-to-evaluate approximation of $\obj$. 
We assume that $\approxfunc$ is differentiable and that each component of $\gradapprox$ is square integrable. 
From now on, we refer to $\obj$ as the \textit{high-fidelity} function and to $\approxfunc$ as the \textit{low-fidelity} function.
Based on the identity
$$
 H = \expect[\grad(X)\grad(X)^T - \gradapprox(X)\gradapprox(X)^T] + \expect[\gradapprox(X)\gradapprox(X)^T],
$$
we introduce the following unbiased multifidelity estimator (MF estimator).
\begin{definition}[Multifidelity estimator]\label{def:MF_definition}
  Let $m_{1}\geq1$ and $m_{2}\geq1$ be the numbers of gradient evaluations of $\obj$ and $\approxfunc$.
  We define the MF estimator of $\matrixH$ to be:
  \begin{align*}
    \HMF &= \frac{1}{m_{1}}\sum_{i=1}^{m_{1}}(\grad(\randd_{i})\grad(\randd_{i})^{T}-\gradapprox(\randd_{i})\gradapprox(\randd_{i})^{T} ) 
    +\frac{1}{m_{2}}\sum_{i=m_{1}+1}^{m_{1}+m_{2}}\gradapprox(\randd_{i})\gradapprox(\randd_{i})^{T},
  \end{align*}
  where $\randd_{1},\ldots,\randd_{m_{1}+m_{2}}$ are independent copies of $X\sim\rho$.
\end{definition}

\begin{remark}[Indefiniteness of $\HMF$]
While the quantity of interest $\matrixH$ is symmetric positive semi-definite, the multifidelity estimator $\HMF$ is symmetric but not necessarily positive semi-definite.
It is natural to ask whether a positive semi-definite estimator is necessary to yield good dimension reduction.
In the following, we show that the quality of the dimension reduction is controlled by the error between $\matrixH$ and $\HMF$ (Corollary~\ref{prop:ControlHatUcorollary}) which can be reduced arbitrarily close to zero with high probability (Proposition~\ref{prop:MainResult}).
In particular, those results do not require positive semi-definiteness from the estimator $\HMF$.
\end{remark}

A realization of $\HMF$ can be obtained using Algorithm~\ref{alg:MFMonteCarlo}. 
First, $m_{1}+m_{2}$ input parameter realizations are drawn independently from $\rho$. 
Then, the high-fidelity gradients are evaluated at the first $m_{1}$ input parameter values while the low-fidelity gradients are evaluated at all $m_{1}+m_{2}$ input parameter values.

\begin{algorithm}[H]
\caption{Multifidelity Active Subspace}
\label{alg:MFMonteCarlo}
\begin{algorithmic}
\STATE {\bfseries Function:} \texttt{multifidelity\_active\_subspace}$(m_{1}, m_{2})$
\STATE {\bfseries Input:} $m_{1}$ and $m_{2}$
\STATE Draw $m_1+m_2$ independent copies $\{X_{i}\}_{i=1}^{m_1+m_2}$ of $X\sim \rho$ 
\FOR{$i=1$ {\bfseries to} $m_{1}$}
\STATE Compute $\grad(X_{i})$ and $\gradapprox(X_{i})$
\ENDFOR
\STATE $\HMF\gets \frac{1}{m_{1}}\sum_{i=1}^{m_{1}}(\grad(X_{i})\grad(X_{i})^{T}-\gradapprox(X_{i})\gradapprox(X_{i})^{T})$
\FOR{$i=1$ {\bfseries to} $m_{2}$}
\STATE Compute $\gradapprox(X_{m_{1}+i})$
\ENDFOR
\STATE $\HMF\gets \HMF + \frac{1}{m_{2}}\sum_{i=m_{1}+1}^{m_{1}+m_{2}}\gradapprox(X_{i})\gradapprox(X_{i})^{T}$
\STATE {\bfseries Output: $\HMF$}
\end{algorithmic}
\end{algorithm}

The definition of the proposed MF estimator of the AS matrix uses the low-fidelity gradient to construct a control variate $\gradapprox(X)\gradapprox(X)^{T}$ for $\grad(X)\grad(X)^{T}$.
The MF estimator is written as the sum of two terms.
The first one involves $m_{1}$ evaluations of the low-fidelity and high-fidelity gradients.
This is an expensive quantity to compute, so the number of samples $m_{1}$ is typically set to a low value.
Note that if $\gradapprox$ is a good approximation of $\grad$,
then the control variate $\gradapprox(X)\gradapprox(X)^{T}$ is highly correlated with $\grad(X)\grad(X)^{T}$ 
and the first term of the estimator has low variance (in a sense yet to be made precise for matrices).
The low variance of $\grad(X)\grad(X)^T - \gradapprox(X)\gradapprox(X)^T$ allows for a good estimator of $\expect[\grad(X)\grad(X)^T - \gradapprox(X)\gradapprox(X)^T]$ despite the small number of samples $m_{1}$.
The second term involves $m_{2}$ evaluations of the cheap low-fidelity gradient.
Thus, $m_{2}$ can usually be set to a large value, allowing for a good estimation of $\expect[\gradapprox(X)\gradapprox(X)^T]$ despite the possibly large variance of $\gradapprox(X)\gradapprox(X)^T$.
Combining the two terms, the MF estimator $\HMF$ provides a good approximation of $\matrixH$ with few evaluations of the expensive high-fidelity gradient $\grad$.
In Section \ref{sec:theoretical_analysis}, we make this statement precise by providing an analysis of the error between $\matrixH$ and $\HMF$ as a function of the number of samples $m_{1}$ and $m_{2}$.

\subsection{Relationship to function approximation} 
\label{sub:relationship_to_function_approximation}

The  performance of our MF estimator (or that of any estimator for $H$) should be analyzed with respect to the end goal of the problem which, in this paper, is to perform dimension reduction.
Computing a good approximation of $\matrixH$ is an \textit{intermediate step} in the dimension reduction process.
To further motivate the use of a MF estimator to reduce the difference between $\matrixH$ and $\HMF$ at low cost, we show how this matrix error impacts the quality of the dimension reduction.
As shown in Section \ref{sec:active_subspace},
one way of performing dimension reduction is to minimize a bound on the function approximation error \eqref{eq:AScontrol}. 
This corresponds to maximizing $U_{r}\mapsto \trace(U_{r}^{T} \matrixH U_{r})$.
Replacing the unknown $\matrixH$ by $\HMF$, we can compute the matrix $\widehat{U}_{r}$ defined by
\begin{equation}\label{eq:ASmaxtraceHMF}
 \widehat U_{r} \in \underset{\substack{U_{r}\in\mathbb{R}^{\dimension\times r} \\ \text{s.t. } U_{r}^{T} U_{r} = I_{r}}} {\text{argmax}}
 \trace(U_{r}^{T} \HMF U_{r}),
\end{equation}
and ask how does $\trace(\widehat U_{r}^{T} \matrixH \widehat U_{r})$ compare to the maximal value of $\trace(U_{r}^{T} \matrixH U_{r})$ over all $U_{r}\in\mathbb{R}^{\dimension\times r}$ such that $U_{r}^{T} U_{r} = I_{r}$.
By definition, we have the inequality in the following direction
$$
 \max_{\substack{U_{r}\in\mathbb{R}^{\dimension\times r} \\ \text{s.t. } U_{r}^{T} U_{r} = I_{r}}} 
 \trace(U_{r}^{T} \matrixH U_{r})
 \geq \trace(\widehat U_{r}^{T} \matrixH \widehat U_{r}).
$$
The next proposition shows that the difference between the two terms of the previous inequality can be controlled by means of the error $\|H-\HMF\|$, where $\Vert\cdot\Vert$ denotes the matrix operator norm.
Note that the proof is not restricted to the MF estimator: the same result holds for any symmetric estimator of $H$.

\begin{proposition}\label{prop:ControlHatUgeneral}
 Let $\widehat{H}$ be a symmetric estimator of $\matrixH$
 and 
 \begin{equation}\label{eq:ASmaxtraceHMFgeneral}
 \widetilde{U}_{r}\in \underset{\substack{U_{r}\in\mathbb{R}^{\dimension\times r} \\ \text{s.t. } U_{r}^{T} U_{r} = I_{r}}} {\text{argmax}}
 \trace(U_{r}^{T} \widehat{H} U_{r}),
 \end{equation}
  then
\begin{equation}\label{eq:ControlHatUgeneral}
    \trace(\widetilde{U}_{r}^{T} \matrixH \widetilde{U}_{r}) 
  \geq
    \max_{\substack{U_{r}\in\mathbb{R}^{\dimension\times r} \\ \text{s.t. } U_{r}^{T} U_{r} = I_{r}}} 
    \trace(U_{r}^{T} \matrixH U_{r})
    -2r\|\matrixH-\widehat{H} \| .
 \end{equation}
\end{proposition}

\begin{proof}
 Consider the eigenvalue decomposition of $\matrixH-\Hhat = V\Sigma V^T$, where $\Sigma=\text{diag}\{\lambda_1,\hdots,\lambda_{\dimension}\}$ is a diagonal matrix containing the eigenvalues of $\matrixH-\Hhat$ and $V\in\mathbb{R}^{\dimension\times\dimension}$ is a unitary matrix. 
 For any matrix $U_r\in\mathbb{R}^{\dimension\times r}$ such that $U_r^TU_r=I_r$, we have
 \begin{align}
  | \trace( U_r^T \matrixH  U_r)  - \trace( U_r^T \Hhat  U_r) | 
  &= |\trace( U_r^T (\matrixH-\Hhat)  U_r)| \nonumber\\
  &= |\trace(\Sigma V^T  U_r  U_r^T V )|  \nonumber\\
  &\leq \max\{|\lambda_1|,\hdots,|\lambda_{\dimension}|\} ~ |\trace( V^T  U_r  U_r^T V )| \nonumber\\
  &= \|\matrixH-\Hhat\| ~ |\trace( U_r  U_r^T )| \nonumber\\
  &= r \|\matrixH-\Hhat\| . \label{eq:tmp185713}
 \end{align}
 Letting $U_r=\widetilde U_r$ in the above relation yields
 \begin{align*}
  \trace( \widetilde U_r^T \matrixH  \widetilde U_r) 
  &\overset{\eqref{eq:tmp185713}}{\geq} \trace( \widetilde U_r^T \Hhat  \widetilde U_r) - r \|\matrixH-\Hhat\| \\
  &\overset{\eqref{eq:ASmaxtraceHMFgeneral}}{\geq} \trace( U_r^T \Hhat  U_r) - r \|\matrixH-\Hhat\| \\
  &\overset{\eqref{eq:tmp185713}}{\geq} \trace( U_r^T \matrixH  U_r) - 2 r \|\matrixH-\Hhat\| .
 \end{align*}
 Maximizing over $U_r\in\mathbb{R}^{\dimension\times r}$, with $U_r^TU_r=I_r$, yields \eqref{eq:ControlHatUgeneral} and concludes the proof.
\end{proof}

\begin{corollary}\label{prop:ControlHatUcorollary}
Let $\HMF$ be a MF estimator of $\matrixH$ and $\widehat{U}_{r}$ be defined by \eqref{eq:ASmaxtraceHMF}. Then
 \begin{equation}\label{eq:ControlHatUcorollary}
    \trace(\widehat U_{r}^{T} \matrixH \widehat U_{r}) 
  \geq
    \max_{\substack{U_{r}\in\mathbb{R}^{\dimension\times r} \\ \text{s.t. } U_{r}^{T} U_{r} = I_{r}}} 
    \trace(U_{r}^{T} \matrixH U_{r})
    -2r\|\matrixH-\HMF \| .
 \end{equation}
\end{corollary}
\begin{proof}
 This follows from applying Proposition~\ref{prop:ControlHatUgeneral} to $\Hhat=\HMF$ and $\widetilde{U}_{r}=\widehat{U}_{r}$.
\end{proof}

We now establish the connection between  the result of Corollary~\ref{prop:ControlHatUcorollary} and the quality of the dimension reduction.
Assume that inequality \eqref{eq:AScontrol} holds true for any $U_r^TU_r=I_d$ (this is in particular the case if $X\sim\mathcal{N}(0,I_d)$).
Replacing $U_r$ by $\widehat U_r$ in \eqref{eq:AScontrol} and using Corollary~\ref{prop:ControlHatUcorollary}, we can write
\begin{align}\label{eq:functionalerrorbound}
 \expect [ ( \obj(X) - h(\widehat U_{r}^{T}X) )^{2} ] 
 &\leq  \trace(H) -  \max_{\substack{U_{r}\in\mathbb{R}^{\dimension\times r} \\ \text{s.t. } U_{r}^{T} U_{r} = I_{r}}}   \trace(U_{r}^{T} \matrixH U_{r}) + 2r \|H-\HMF\|
 \\ & = \big( \lambda_{r+1} + \hdots + \lambda_d \big) + 2r \|H-\HMF\| ,
\end{align}
where $h(\widehat U_{r}^{T}X) = \expect[f(X)|\widehat U_{r}^{T}X]$. Here $\lambda_i\geq0$ denotes the $i$-th largest eigenvalue of $H$.
This relation shows that a strong decay in the spectrum of $H$ is favorable to efficient dimension reduction.
Also, increasing the number $r$ of active variables has competitive effects on the two terms in the right-hand side: the first term $(\lambda_{r+1} + \hdots + \lambda_d)$ is reduced whereas the second term $2r \|H-\HMF\|$ increases linearly in $r$.
Given the importance of $\|\matrixH-\HMF\|$ in controlling the quality of the dimension reduction, we show in the next section how this error can be controlled at cheap cost using the proposed MF estimator.

\begin{remark}[Angle between subspaces]
 Another way of controlling the quality of the approximate active subspace (the span of the columns of $\widehat U_{r}$) is via the \emph{principal angle} between the exact and the approximate active subspaces \cite{holodnak2018probabilistic,constantine2014computing}. 
 This angle, denoted by $\angle(\widehat{\mathcal{S}},\mathcal{S})$, is defined by
 $$
  \sin( \angle(\widehat{\mathcal{S}},\mathcal{S}) )= \| \widehat U_r \widehat U_r^T -  \tilde{U}_r \tilde{U}_r^T \| ,
 $$
 where $\widehat{\mathcal{S}}=\text{range}(\widehat U_r)$ is the approximate subspace and $\mathcal{S}=\text{range}(\tilde{U}_r)$ is the exact active subspace, $\tilde{U}_r$ being a solution to \eqref{eq:ASmaxtrace}.
 This requires $\widehat{\mathcal{S}}$ and $\mathcal{S}$ to be uniquely defined, which might not be the case if there is a plateau in the spectra of $\matrixH$ and $\HMF$.
 For instance, if the $r$-th eigenvalue of $H$ equals the $(r+1)$-th, problem \eqref{eq:ASmaxtrace} admits infinitely many solutions and $\mathcal{S}$ is not uniquely defined.
 Note that in practice, $r$ is chosen such that the spectral gap is large, by inspection of the spectrum.
 To our knowledge, all analyses focusing on controlling the principal angle $\angle(\widehat{\mathcal{S}},\mathcal{S})$ rely on the spectral gap assumption $\lambda_{r}>\lambda_{r+1}$, where $\lambda_r$ is the $r$-th eigenvalue of $H$.
 
 In contrast, the goal-oriented approach consisting of minimizing the upper bound of the functional error does not require the spectral gap assumption.
 This results from \eqref{eq:ASmaxtrace} and Corollary \ref{prop:ControlHatUcorollary}.
 In particular, the uniqueness of the active subspace is not required, as any solution to \eqref{eq:ASmaxtrace} yields an equally good active subspace for the purpose of function approximation.
 Therefore, in this paper, we do not further consider the principal angle $\angle(\widehat{\mathcal{S}},\mathcal{S})$.
 Instead we focus on comparing $\trace(\widehat{U}_{r}^{T} \matrixH \widehat U_r)$ to $\trace(\tilde{U}_{r}^{T} \matrixH \tilde{U}_{r})$.
 As illustrated by Corollary \ref{prop:ControlHatUcorollary}, this is sufficient to control the error 
 $\Vert\matrixH - \HMF\Vert$
 between the AS matrix and its estimator.
\end{remark}

\section{A non-asymptotic analysis of the estimator}
\label{sec:theoretical_analysis}
In this section, we use results from non-asymptotic random matrix theory to  express the number of gradient evaluations sufficient to control the
error in an estimate of $\matrixH$, up to a user-defined tolerance. We present our main results in this Section and defer the proofs to Appendix~\ref{sec:MainProof} and Appendix~\ref{sec:SingleFidelityResult}.

In general, Monte Carlo estimation of a high-dimensional object such as a $\dimension\times\dimension$ matrix can require a large number of samples. If the matrix does not have some special structure, one can expect the sample complexity to scale with the large ambient dimension $\dimension$. This is costly if each sample is expensive.
However, the AS matrix $\matrixH$ enjoys some structure when the problem has low effective dimension.
In particular, when most of the variation of $\obj$ is concentrated in a low-dimensional subspace,
we expect the number of samples required to obtain a good approximation of $\matrixH$ to depend on the dimension of this subspace, rather than on the ambient dimension $\dimension$ of the input space.
One case of interest occurs when $\obj$ is a ridge function that \textit{only} depends on a small number of linear combinations of input variables. This leads to a rank-deficient matrix $\matrixH$.
In such a case, we expect the number of samples to depend on the rank of $\matrixH$.
Another important case occurs when a (possibly full-rank)  matrix $\matrixH$ has a quickly decaying spectrum.
In such a case, we expect that the number of samples should depend on a characteristic quantity of the spectrum (e.g., the sum of the eigenvalues).
To make a precise statement, we use the notion of \textit{intrinsic dimension} \cite{tropp_intro_MAL-048} (Def.~7.1.1), also called the \textit{effective rank} \cite{koltchinskii2016asymptotics} (Def.~1). The intrinsic dimension of $\matrixH$ is defined by
\begin{align*}
 \intdim = \frac{\trace(H)}{\|\matrixH\|}.
\end{align*}
The intrinsic dimension is a measure of the spectral decay of $\matrixH$. 
It is bounded by the rank of $\matrixH$, i.e., $1\leq \intdim\leq\text{rank}(\matrixH) \leq\dimension$.

Our main result, Proposition \ref{prop:MainResult} below, establishes how many evaluations of the gradient are sufficient to guarantee that the error $\Vert\matrixH-\HMF\Vert$ is below some user-defined tolerance.
In particular, the number of gradient evaluations from the low-fidelity and high-fidelity models is shown to be a function of the intrinsic dimension of $\matrixH$ and of two coefficients $\theta$ and $\beta$, 
characterizing the quality of the low-fidelity model and the maximum relative magnitude of the high-fidelity gradient, respectively.
\begin{proposition}\label{prop:MainResult}
 Assume there exist positive constants $\beta <\infty$ and $\theta <\infty$ such that the relations
 \begin{align}
  \Vert\grad(X)\Vert^{2} &\leq \beta^{2} \, \normE , \label{eq:Assumption BETA}\\
  \Vert \grad(X) - \gradapprox(X) \Vert^{2} &\leq \theta^{2} \, \normE , \label{eq:Assumption THETA}
 \end{align}
 hold almost surely.
 Let $\HMF$ be the MF estimator introduced in Definition \ref{def:MF_definition} and assume
 \begin{align}\label{eq:m2geqm1}
   m_2\geq m_1 \max\left\{\frac{(\theta+\beta)^2(1+\theta)^2}{\theta^{2}(2+\theta)^{2}} \,;\, \frac{(\theta+\beta)^{2}}{\theta(2\beta+\theta)}\right\}.
 \end{align}
 Then, for any $\varepsilon>0$, the condition
 \begin{align}\label{eq:m1geq MEAN}
    m_1 \geq \varepsilon^{-2} \intdim \, \theta \, \log(2d) \max \big\{ 4\intdim\theta(2+\theta)^2 \,;\, 2/3(2\beta+\theta) \big\} ,
 \end{align}
 is sufficient to ensure
 \begin{align}\label{eq:MainResult MEAN}
   \expect[\Vert\matrixH-\HMF\Vert]\leq (\varepsilon+\varepsilon^{2})\normH.  
 \end{align}
 Furthermore for any $0<\varepsilon\leq1$ and $0<\eta<1$, the condition
 \begin{align}\label{eq:m1geq PROBA}
   m_1\geq \varepsilon^{-2} \intdim \, \theta\, \log(2d / \eta) \, \big( 4\intdim\theta(2+\theta)^2+ \varepsilon 4/3(2\beta+\theta) \big),
 \end{align}
 is sufficient to ensure
 \begin{align}\label{eq:MainResult PROBA}
   \proba\big\{\Vert\matrixH-\HMF\Vert \leq \varepsilon\normH \big\} \geq 1-\eta.
 \end{align}
\end{proposition}
\begin{proof}
  See Appendix \ref{sec:MainProof}.
\end{proof}

Similarly, we can derive the number of high-fidelity gradient evaluations $m_{1}$ sufficient to control the SF estimator error in expectation and with high probability. 
This is the purpose of the following proposition (see also \cite{lam2017thesis}).
Note that a high-probability bound similar to equations \eqref{eq:m1geq SF PROBA} and \eqref{eq:SingleFidelityResult PROBA} is established by Corollary 2.2 from \cite{holodnak2018probabilistic}.
\begin{proposition}\label{prop:SingleFidelityResult}
 Assume there exists $\beta <\infty$ such that
 \begin{align}\label{eq:Assumption BETA SF}
  \Vert\grad(X)\Vert^{2} &\leq \beta^{2} \, \normE , 
 \end{align}
 holds almost surely. Then for any $0<\varepsilon\leq1$ the condition
 \begin{align}\label{eq:m1geq SF}
  m_1 \geq C \varepsilon^{-2}\intdim \log(1+2\intdim) (1+\beta^2),
 \end{align}
 is sufficient to ensure
 \begin{align}\label{eq:SingleFidelityResult MEAN}
  \expect[\Vert \matrixH-\HSF \Vert] \leq (\varepsilon+\varepsilon^2)\|\matrixH\| .
 \end{align}
 Here $C$ is an absolute (numerical) constant.
 Furthermore for any $0<\varepsilon\leq1$ and any $0<\eta\leq1$, the condition
 \begin{align}\label{eq:m1geq SF PROBA}
  m_1 \geq 2 \varepsilon^{-2}\intdim \log(8\intdim/\eta) ( \beta^2 + \varepsilon(1+\beta^2)/3) ,
 \end{align}
 is sufficient to ensure
 \begin{align}\label{eq:SingleFidelityResult PROBA}
  \proba\{\Vert \matrixH-\HSF \Vert \leq \varepsilon\|\matrixH\| \} \geq 1-\eta.
 \end{align}
\end{proposition}
\begin{proof}
  See Appendix \ref{sec:SingleFidelityResult}.
\end{proof}

The previous propositions show that when $\theta^{2}$ is small (i.e., when $\gradapprox$ is a good approximation of $\grad$), the number of samples $m_1$ of the high-fidelity function can be significantly reduced compared to the single-fidelity approach.
The number of samples $m_2$ has to be adjusted according to 
\eqref{eq:m2geqm1}.
The guarantees provided for the MF estimator are different than those for the SF estimator.
For the MF estimator, the bound on $m_1$ is a function of the intrinsic dimension but is also weakly (i.e., logarithmically) dependent on the ambient dimension $\dimension$.

These results are especially interesting when $\beta^{2}$ has no dependency (or weak dependency) on the ambient dimension $\dimension$.
In such a case, the number of evaluations sufficient to obtain a satisfactory relative error depends only on the intrinsic dimension $\intdim$ and the parameter $\beta^{2}$, and does not depend (or only weakly depends) on the ambient dimension $\dimension$.
Recall that $\beta^{2}$ quantifies the variation of the square norm of the gradient, $\Vert \grad\Vert^{2}$, relative to its mean $\normE$.
In Section~\ref{sec:numerical_results}, we provide an example of gradient function $\grad$ for which $\beta^{2}$ is independent of $\dimension$.

The proofs of Propositions \ref{prop:MainResult} and \ref{prop:SingleFidelityResult} use a similar strategy.
The key ingredient is the use of a concentration inequality to bound the error between the AS matrix and its estimator.
These inequalities are applicable to matrices expressed as the sum of independent matrices (i.e., independent summands), a condition met by the MF and SF estimators of $\matrixH$.
The error bounds depend on two characteristics of the matrix of interest: the variance of the estimator and an upper bound on the norm of the summands.
Those two characteristic quantities are functions of the number of samples used to construct the estimator.
Once established, those bounds are used to express sufficient conditions on the number of samples to guarantee a user-defined tolerance for the error, both in expectation and with high probability.
The full proofs of Propositions \ref{prop:MainResult} and \ref{prop:SingleFidelityResult} are given in Appendix~\ref{sec:MainProof} and Appendix~\ref{sec:SingleFidelityResult}.

We conclude this section by summarizing the connection between our main results and a quantity of interest in dimension reduction: the functional error.
As shown in \eqref{eq:AScontrol}, for any matrix $U_{r}\in\mathbb{R}^{\dimension \times r}$ with $r\leq\dimension$ orthonormal columns, the functional error $\expect [ ( \obj(X) - h(U_{r}^{T}X) )^{2}]$ is upper bounded by 
$\trace(H) - \trace(U_{r}^{T} \matrixH U_{r})$. Thus, the quality of the dimension reduction can be controlled by finding the maximizer $U_{r}^{*}$ of $U_{r}\mapsto \trace(U_{r}^{T} \matrixH U_{r})$, as $U_{r}^{*}$ yields the tightest upper bound on the functional error.
However, $U_{r}^{*}$ cannot be computed because $\matrixH$ is unknown.
Instead, we compute an approximator $\HMF$ of $\matrixH$ and its associated $\widehat{U}_{r}$ (see \eqref{eq:ASmaxtraceHMF}).
Corollary~\ref{prop:ControlHatUcorollary} shows that using $\widehat{U}_{r}$ instead of $U_{r}^{*}$ yields an upper bound close to the tightest one when $\Vert\matrixH-\HMF\Vert$ is small.
As a result, the functional error $\expect [ ( \obj(X) - h(\widehat{U}_{r}^{T}X) )^{2}]$ incurred by the ridge approximation built with $\widehat{U}_{r}$ is at most $2r\Vert\matrixH-\HMF\Vert$ larger than the tightest bound defined by the unknown $U_{r}^{*}$ (see \eqref{eq:functionalerrorbound}).
Finally, Proposition~\ref{prop:MainResult} shows how $\Vert\matrixH-\HMF\Vert$ can be controlled by increasing the number of gradient evaluations.
Therefore, those results establish a direct link between the number of gradient evaluations and the upper bound on the functional error of the ridge function built with $\widehat{U}_{r}$.

\section{Numerical results}
\label{sec:numerical_results}
In this section, we conduct numerical experiments illustrating the performance of the proposed MF estimator. 
We first demonstrate the algorithm on synthetic examples for which we can compute errors; we conduct a parametric study and compare the SF and MF estimator performances. 
Second, we consider a high-dimensional engineering case: performing dimension reduction on a linear elasticity problem involving parameterized material properties of a wrench. 
Finally, we consider an expensive engineering problem: finding the active subspaces associated with the shape optimization of the ONERA M6 wing in a turbulent flow.

\subsection{Analytical problem} 
\label{sub:analytical_problem}

In this section, we consider an example for which all characteristic quantities 
($\intdim$,~ $\normH$,~ $\normE$,~ $\beta^{2}$) are known with closed-form expressions.
\footnote{Code for the analytical problems available at \url{https://github.mit.edu/rlam/MultifidelityDimensionReduction}.} 

\subsubsection{Problem description}
\label{sub:description_analytic}

We consider the input space $\inputspace = [-1,1]^{\dimension}$ and define $\rho$ to be the uniform distribution over $\inputspace$. 
We introduce the function $\obj:\inputspace\rightarrow \mathbb{R}$ such that for all $\design\in\inputspace$
$$
 f(\design) = \frac{\sqrt{3}}{2} \sum_{i=1}^{\dimension} a_i \, x_i^2 
 \quad\text{ and }\quad
 \grad(\design) = \sqrt{3}
 \begin{pmatrix}
  a_1 x_1 \\ \vdots \\ a_{\dimension} x_{\dimension}
 \end{pmatrix},
$$
where $\boldsymbol{a}=(a_1,\hdots,a_{\dimension})^{T}\in\mathbb{R}^{\dimension}$ is a user-defined vector with $|a_{1}|\geq |a_{2}| \geq \hdots \geq 0$.
With $X\sim\rho$, we have
$$
 H = \expect[\grad(X)\grad(X)^T]
 = \begin{pmatrix}
    a_{1}^{2} &&0\\
    &\ddots&\\
    0 & & a_{\dimension}^{2}
  \end{pmatrix},
$$
so that $\|H\|=a_1^2$, $\trace(H)= \Vert\boldsymbol{a}\Vert^2$ and 
$$
 \intdim=\frac{\trace(\matrixH)}{\normH} = \frac{\Vert\boldsymbol{a}\Vert^2}{a_1^2}.
$$
We define $\beta$ as the smallest parameter that satisfies $\|\grad(X)\|^2\leq \beta^2 \, \expect[\|\grad(X)\|^2]$.
Given that $\expect[\|\grad(X)\|^2] = \trace(H)=\Vert\boldsymbol{a}\Vert^2$, we have
$$
 \beta^2 
 = \sup_{\design\in\inputspace} \frac{\|\grad(\design)\|^2}{\expect[\|\grad(X)\|^2]}
 = \sup_{\design\in\inputspace} \frac{3 \sum_{i=1}^{\dimension} (a_i x_i)^2}{\sum_{i=1}^{\dimension} a_i^2}
 = 3,
$$
where the supremum is attained by maximizing each term in the numerator.
This corresponds to $x_i^2=1$ for all $1\leq i \leq \dimension$.
Notice that $\beta=\sqrt{3}$ is independent of the ambient dimension $\dimension$ and the user-defined vector $\boldsymbol{a}$.

We define the low-fidelity function $g:\inputspace\rightarrow\mathbb{R}$ such that, for all $\design\in\inputspace$, $g(\design) = \obj(\design) - b T \Vert\boldsymbol{a}\Vert  \cos(x_{\dimension} /T)$, where $b\geq0$ and $T>0$ are two user-defined parameters. 
In other words, $g$ is a perturbation of $\obj$ such that $g-f$ depends only on the last component. We have
$$
 \gradapprox(\design) = \grad(\design) +
 \begin{pmatrix}
  0\\\vdots\\0\\ b \Vert\boldsymbol{a}\Vert \sin(x_{\dimension} /T)
 \end{pmatrix},
$$
for all $\design\in\inputspace$.
We let $\theta$ be the smallest parameter such that $\| \grad(\design) - \gradapprox(\design) \|^2 \leq \theta^2 \normE$ for all $\design\in\inputspace$.
We obtain
\begin{align*}
 \theta^2 
 &= \sup_{\design\in\inputspace} \frac{\| \grad(\design) - \gradapprox(\design) \|^2}{\normE} 
 = \sup_{\design\in\inputspace} \frac{(b \Vert\boldsymbol{a}\Vert \sin(x_{\dimension} /T))^2}{\Vert\boldsymbol{a}\Vert^2} 
 = \begin{cases}
  b^2\sin(1/T)^2 &\text{if } T \geq \frac{2}{\pi}\\ 
  b^2 &\text{otherwise.}
 \end{cases}
\end{align*}
For the numerical experiments, we set $b=\sqrt{0.05}$ and $T=0.1$, leading to a parameter $\theta=\sqrt{0.05}$.
The number of samples $m_{2}$ is set using the criteria of \eqref{eq:m2geqm1}, leading to
$$
 m_2 = 63 m_1 
 ~\geq~
 m_1 \, \max\left\{\frac{(\theta+\beta)^2(1+\theta)^2}{\theta^{2}(2+\theta)^{2}} \,;\, \frac{(\theta+\beta)^{2}}{\theta(2\beta+\theta)}\right\} .
$$
In the two following subsections, we consider the case where $H$ is rank deficient, and the case where $H$ is full rank but has a small intrinsic dimension.
From now on, the ambient dimension is set to $\dimension=100$.

\subsubsection{Rank-deficient matrices} 
\label{sub:rank_deficient_matrices}

In this section, we consider the parameter $\boldsymbol{a}\in\mathbb{R}^{\dimension}$ defined by
$$
 \boldsymbol{a} = ( \underbrace{1,\hdots,1}_{k},\underbrace{0,\hdots,0}_{\dimension-k} ),
$$
for some $k\in\{1,3,10,30,100\}$. 
With this choice, the function $\obj$ only depends on the $k$ first variables of the input $\design$.
This yields
$$
 H = \begin{pmatrix} I_k & 0 \\ 0 & 0 \end{pmatrix}
 \quad\text{and}\quad
 \intdim = k \in\{1,3,10,30,100\},
$$
where $I_k$ is the identity matrix of size $k$.
We study the dependence of the relative error $\|\widehat{\matrixH}-\matrixH\|/\normH$ on the intrinsic dimension $\intdim$ and the number of samples $m_{1}$.

Figure~\ref{fig:rank_def_sample} shows the average relative error of the SF estimator (left panel) and the MF estimator (right panel) as a function of the number of samples $m_{1}$ for $\intdim\in\{1,10,100\}$.
The numerical results are in accordance with the theoretical bounds: the errors are dominated by the theoretical bounds and the slopes are similar.
Comparing the left and right panels, we conclude that the MF estimator outperforms the SF estimator for a given number of high-fidelity evaluations.
For example, in the case $\intdim=1$, with $m_{1}= 10$ high-fidelity samples, the MF estimator achieves a relative error of 0.07 while the SF relative error is 0.23---i.e., a relative error 3.47 times lower.
Similarly, Figure~\ref{fig:rank_def_intdim} shows the average relative error as a function of the intrinsic dimension of $\matrixH$ for $m_{1}\in\{10,100,1000\}$.
We notice that the difference between the theoretical bound and the actual estimator error is larger for the MF estimator than for the SF estimator.
\begin{figure}[h]
  \centering
  \begin{subfigure}[t]{0.48\textwidth}
    \includegraphics[width = 1.0\textwidth]{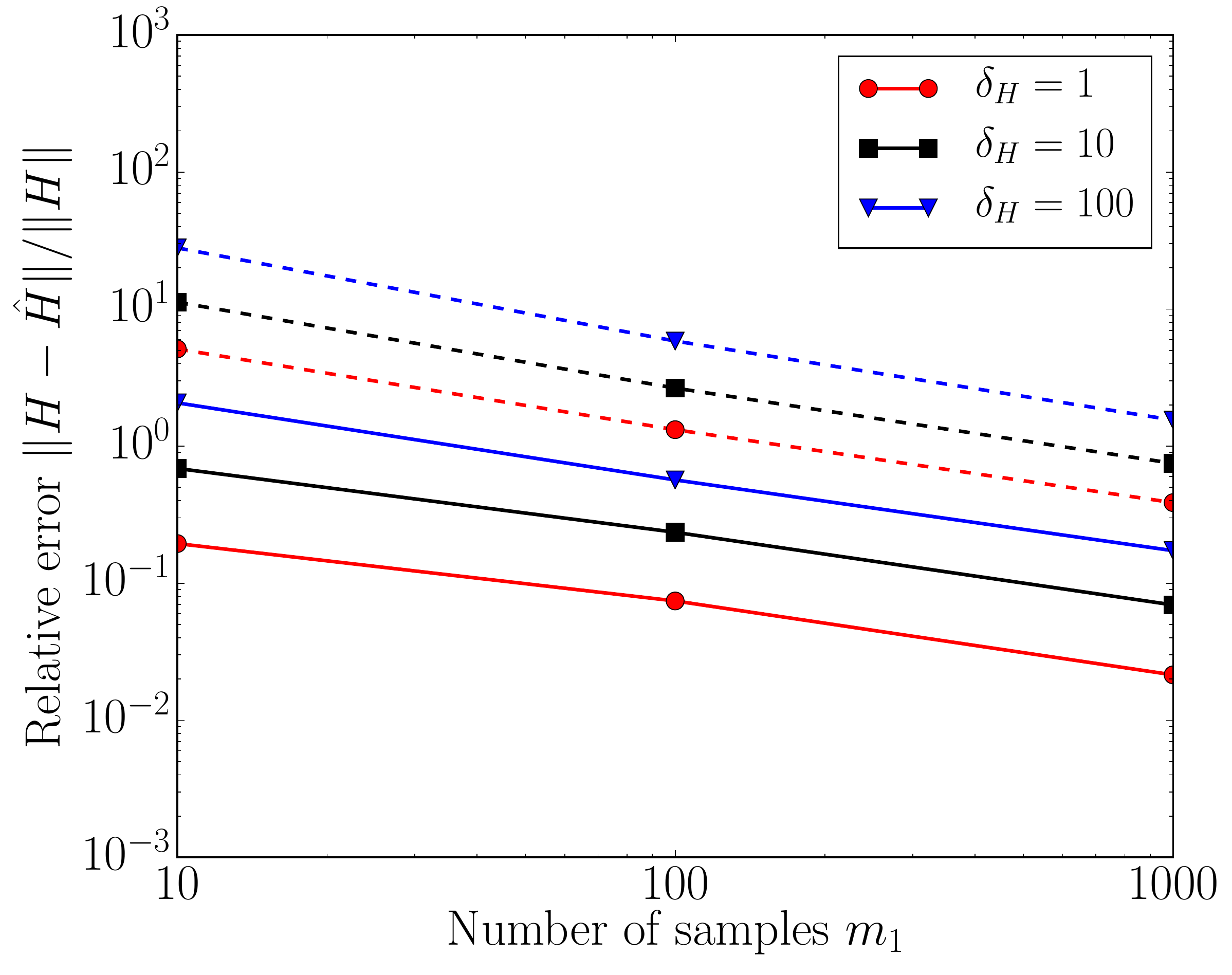}
  \end{subfigure}
  \begin{subfigure}[t]{0.48\textwidth}
    \includegraphics[width = 1.0\textwidth]{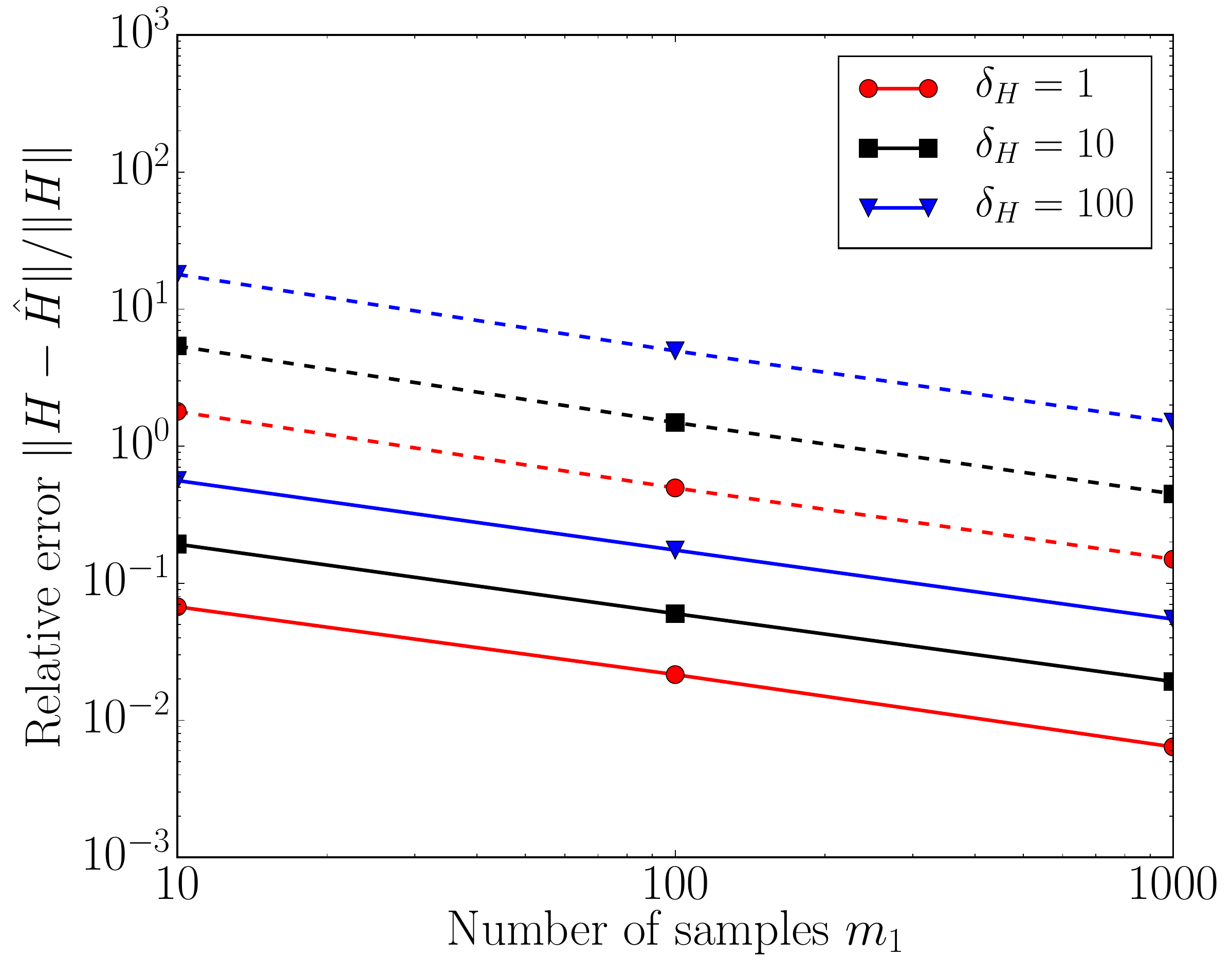}
    \end{subfigure}
  \caption{Rank-deficient example. Average relative error (solid line) as a function of the number of samples $m_{1}$ for the SF estimator (left panel) and the MF estimator (right panel). Average is computed over 100 trials. Dashed lines represent theoretical bounds.}\label{fig:rank_def_sample}
\end{figure}
\begin{figure}[h]
  \centering
  \begin{subfigure}[t]{0.48\textwidth}
    \includegraphics[width = 1.0\textwidth]{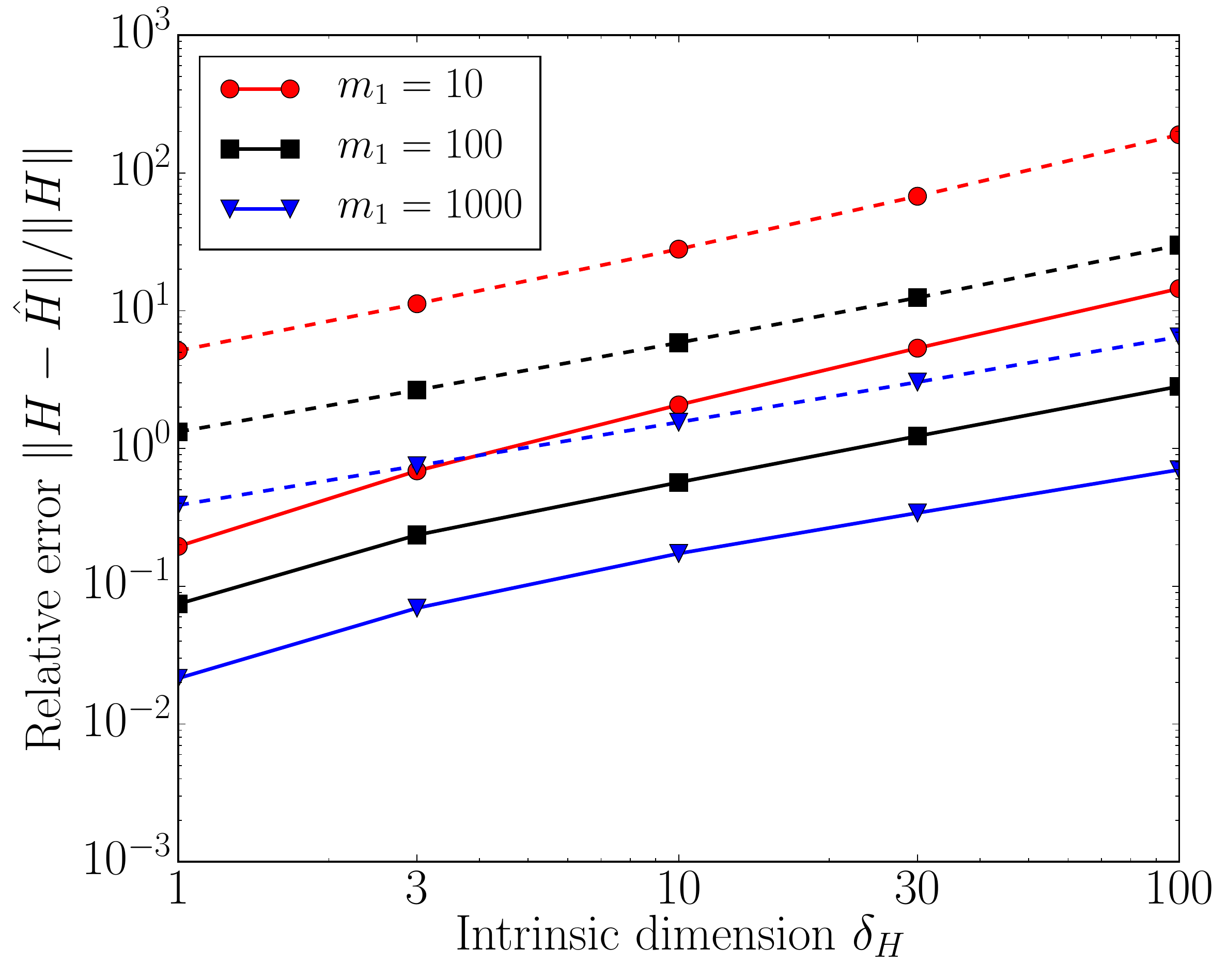}
  \end{subfigure}
  \begin{subfigure}[t]{0.48\textwidth}
    \includegraphics[width = 1.0\textwidth]{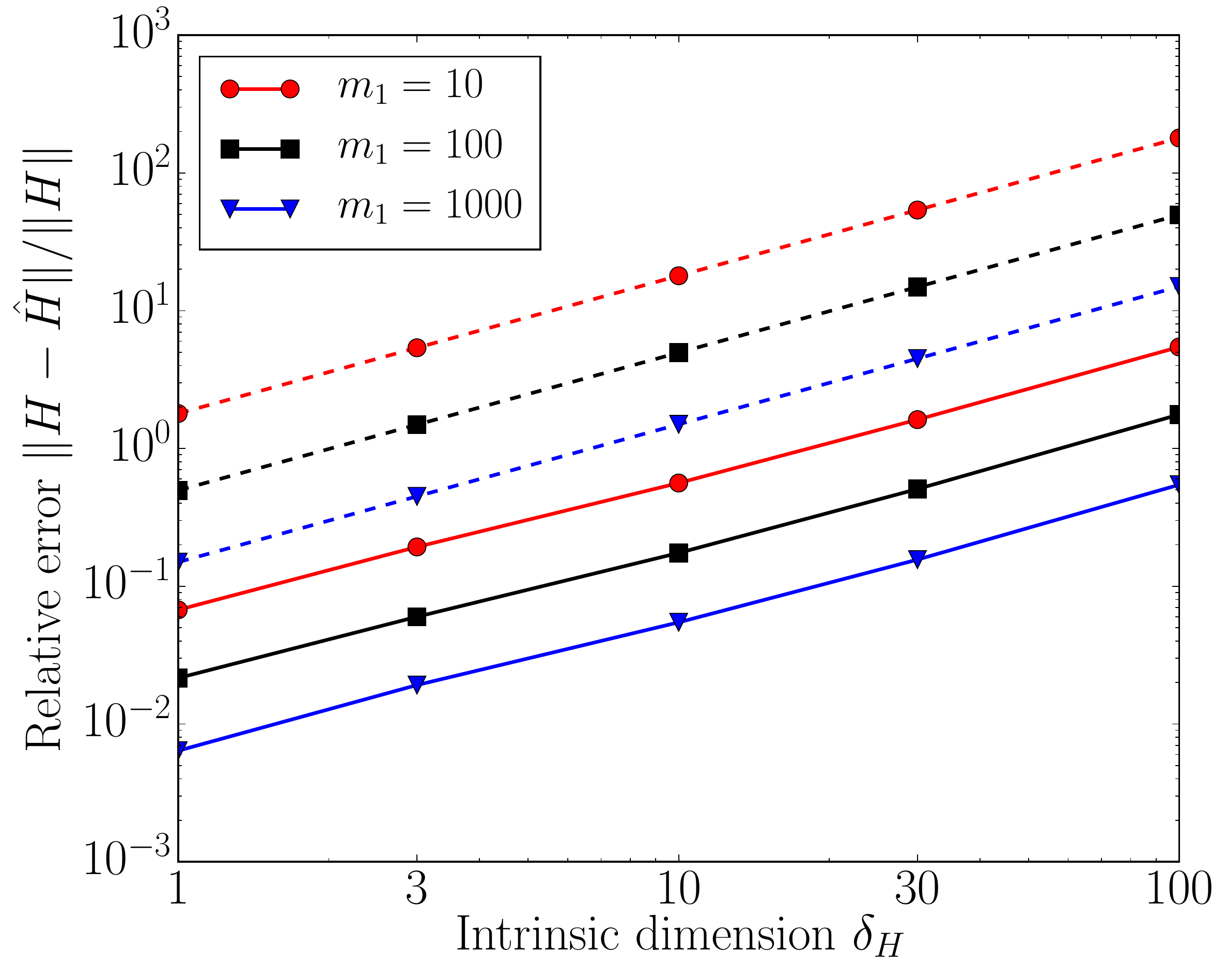}
    \end{subfigure}
  \caption{Rank-deficient example. Average relative error (solid line) as a function of the intrinsic dimension $\intdim$ for the SF estimator (left panel) and the MF estimator (right panel). Average is computed over 100 trials. Dashed lines represent theoretical bounds.}\label{fig:rank_def_intdim}
\end{figure}

\subsubsection{Full-rank matrices} 
\label{sub:full_rank_matrices}
In this section,
 we define the parameter $\boldsymbol{a}\in\mathbb{R}^{\dimension}$ to be
$
 a_i = \exp( - C i ),
$
for all $1\leq i \leq \dimension$, where the value of $C\geq0$ is set to match the intrinsic dimensions 
$
 \intdim\in\{2,5,10,50,100\}.
$
The function $\obj$ now depends on every component of the input $\design$ and the matrix $\matrixH$ is full rank.
Again, we study the dependence of the relative error $\|\widehat{\matrixH}-\matrixH\|/\normH$ on the intrinsic dimension $\intdim$ and the number of samples $m_{1}$.

Figure~\ref{fig:full_rank_sample} shows the average relative error of the SF estimator (left panel) and the MF estimator (right panel) as a function of the number of samples $m_{1}$ for $\intdim\in\{2,10,100\}$.
The numerical results agree with the theoretical bounds: the errors are dominated by the theoretical bounds and the slopes are similar.
Comparing the left and right panels, we conclude that the MF estimator outperforms the SF estimator for a given number of high-fidelity evaluations.
For example, in the case $\intdim=1$, with $m_{1}= 10$ high-fidelity samples, the MF estimator achieves a relative error of 0.09 while the SF relative error is 0.37, leading to a relative error 3.86 times lower.
Similarly, Figure~\ref{fig:full_rank_intdim} shows the average relative error as a function of the intrinsic dimension of $\matrixH$ for $m_{1}\in\{10,100,1000\}$.
Again, we observe that the empirical results satisfy the theoretical bounds.
\begin{figure}[h]
  \centering
  \begin{subfigure}[t]{0.48\textwidth}
      \includegraphics[width = 1.0\textwidth]{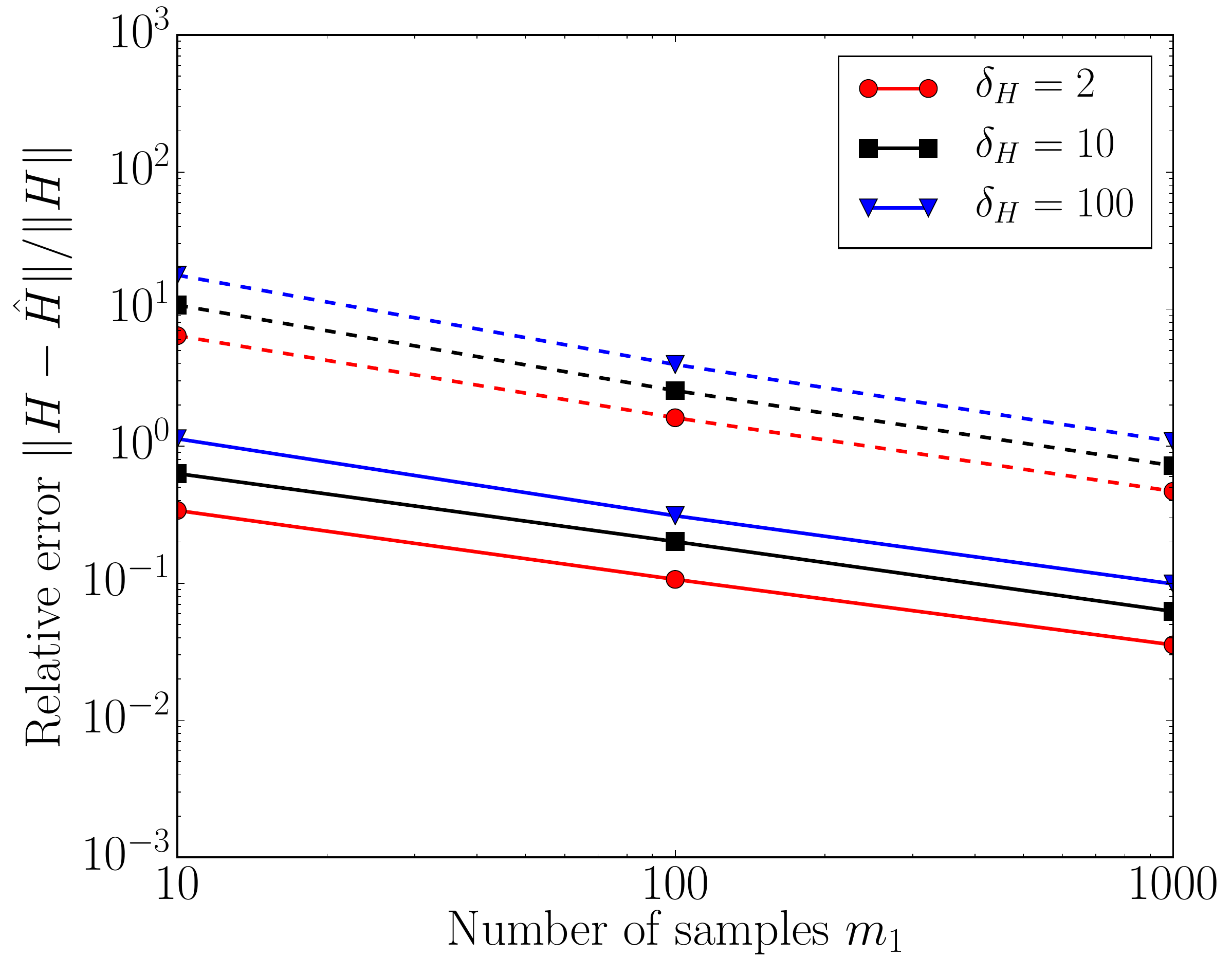}
  \end{subfigure}
  \begin{subfigure}[t]{0.48\textwidth}
    \includegraphics[width = 1.0\textwidth]{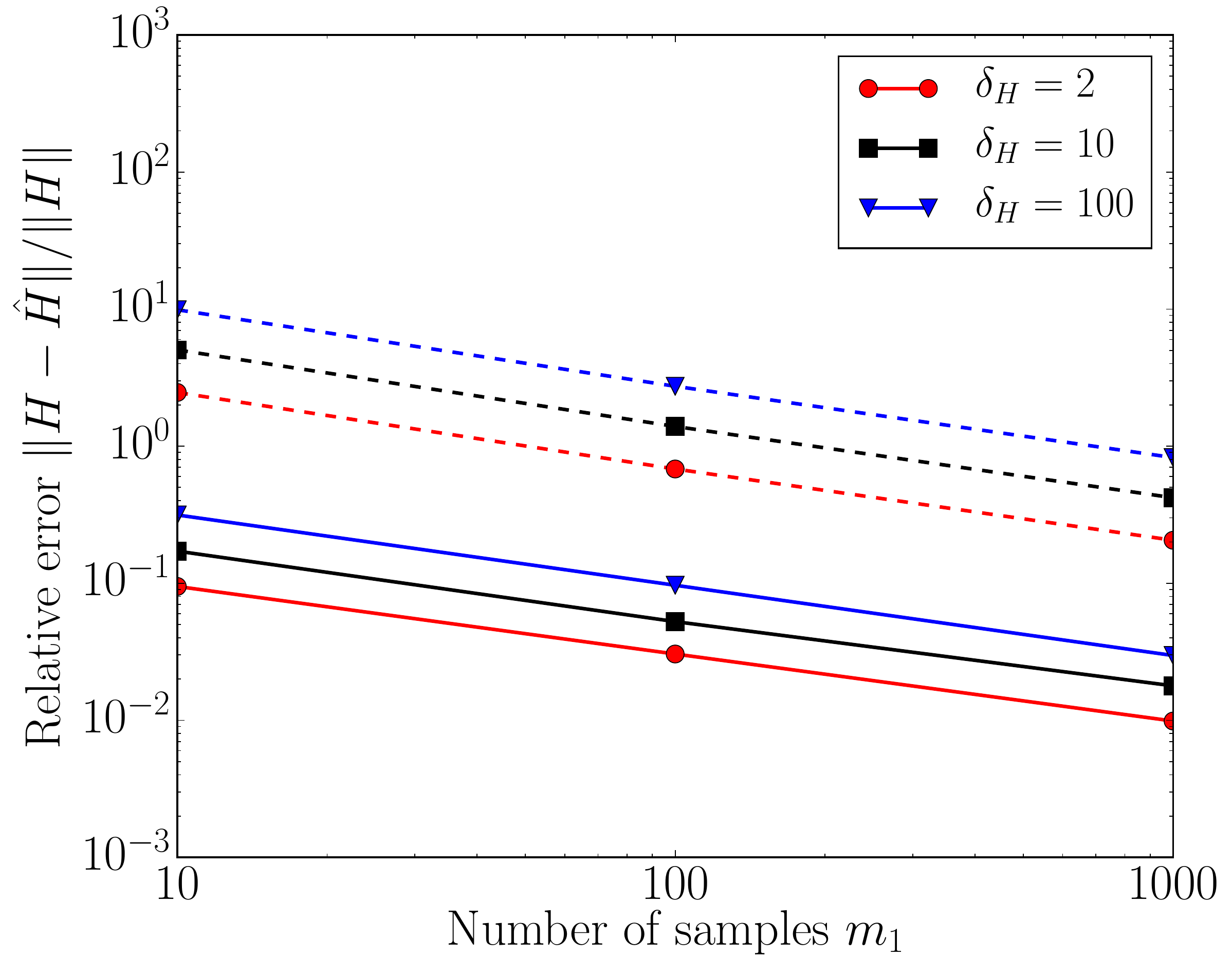}
    \end{subfigure}
  \caption{Full-rank example. Average relative error (solid line) as a function of the number of samples $m_{1}$ for the SF estimator (left panel) and the MF estimator (right panel). Average computed over 100 trials. Dashed lines represent theoretical bounds.}\label{fig:full_rank_sample}
\end{figure}

\begin{figure}[h]
  \centering
  \begin{subfigure}[t]{0.48\textwidth}
    \includegraphics[width = 1.0\textwidth]{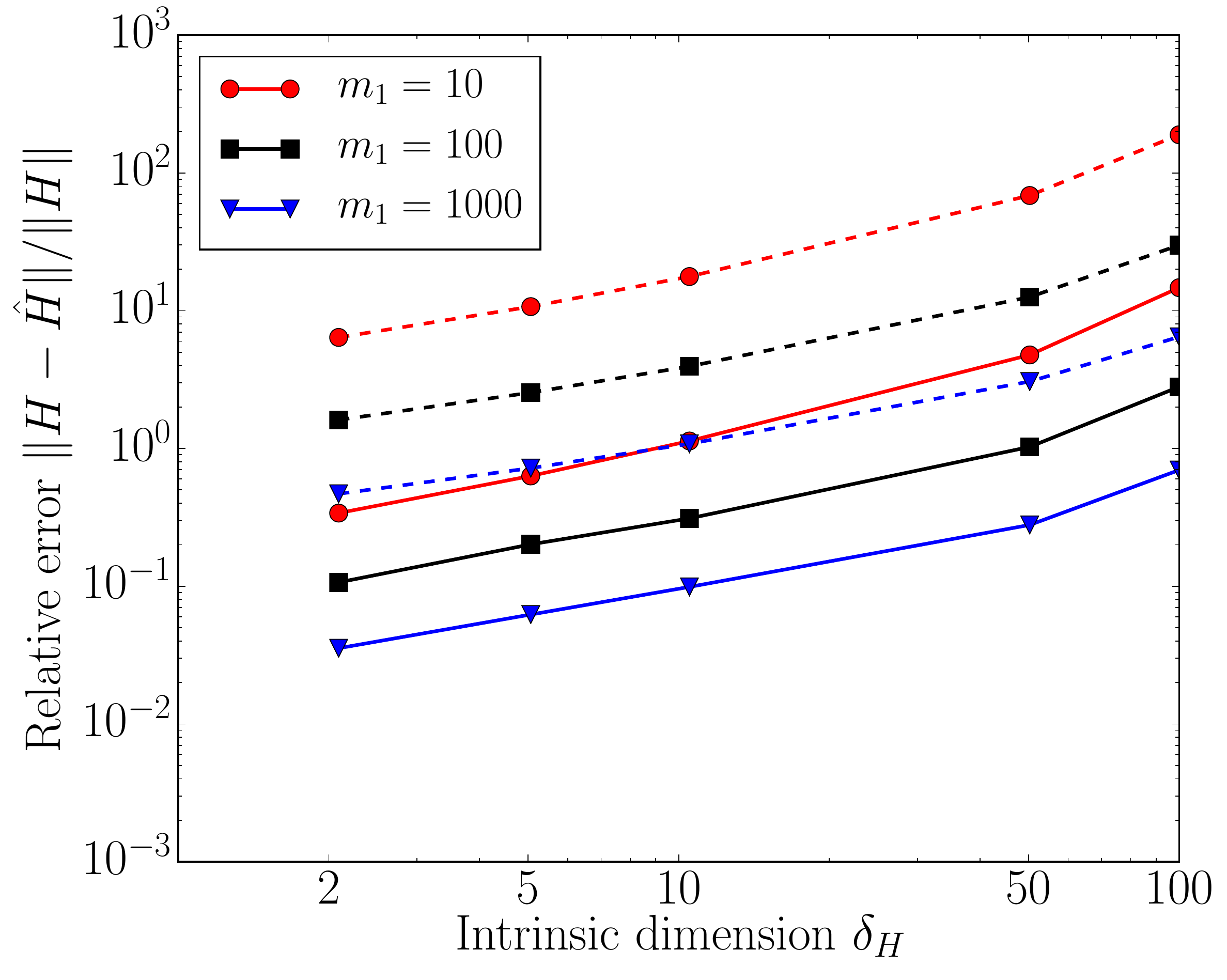}
  \end{subfigure}
  \begin{subfigure}[t]{0.48\textwidth}
    \includegraphics[width = 1.0\textwidth]{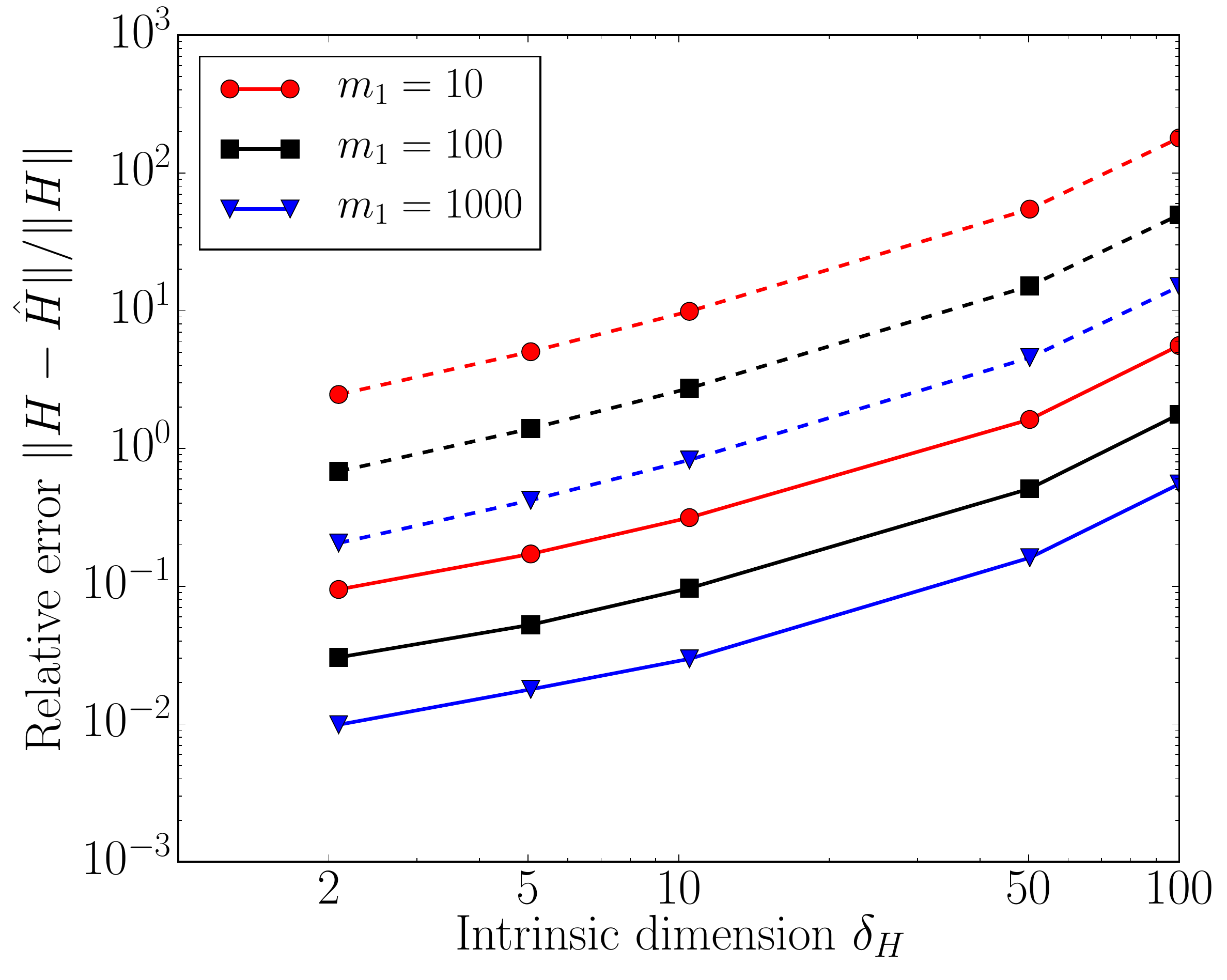}
    \end{subfigure}
  \caption{Full-rank example. Average relative error (solid line) as a function of the intrinsic dimension $\intdim$ for the SF estimator (left panel) and the MF estimator (right panel). Average computed over 100 trials. Dashed lines represent theoretical bounds.}\label{fig:full_rank_intdim}
\end{figure}


\subsection{Linear elasticity analysis of a wrench} 
\label{sub:wrench_material_composition}
We now demonstrate the proposed MF estimator on a high-dimensional engineering example.

\subsubsection{Problem description}

We consider a two-dimensional linear elasticity problem which consists of finding a smooth displacement field $u:\Omega\rightarrow\mathbb{R}^2$ that satisfies
$$
 \text{div} ( K:\varepsilon( u ) ) = 0
 \quad\text{ on } \Omega \subset\mathbb{R}^{2} ,
$$ 
where $\varepsilon(u) = \frac{1}{2}(\nabla u + \nabla u^T)$ is the strain field and $K$ is the Hooke tensor. 
The boundary conditions are depicted in Figure~\ref{fig:Wrench_GEOMETRY_1}.
Using the plane stress assumption, the Hooke tensor $K$ satisfies
$$
 K:\varepsilon( u ) = 
 \frac{E}{1+\nu} \varepsilon( u ) + \frac{\nu E }{1-\nu^2} \trace(\varepsilon( u )) I_2,
$$
where $\nu=0.3$ is the Poisson's ratio and $E\geq0$ the Young's modulus.
We model $E$ as a random field such that $\log(E) \sim \mathcal{N}(0,C)$ is a zero-mean Gaussian field over $\Omega$ with covariance function $C:\Omega\times\Omega\rightarrow\mathbb{R}$ such that $C(s,t)= \exp(-\|s-t\|_2^2/l_0)$ with $l_0=1$.
We consider the output defined by the vertical displacement of a point of interest (PoI) represented by the green point in Figure~\ref{fig:Wrench_GEOMETRY_1}.
We denote by $u_{2}(\mathrm{PoI})$ this scalar output, where the subscript `$2$' refers to the vertical component of $u$.

We denote by $u^{h}$ the finite element solution defined as the Galerkin projection of $u$ on a finite element space comprising continuous piecewise linear functions over the mesh depicted in Figure~\ref{fig:Wrench_GEOMETRY_2}. 
To compute $u^{h}$, we approximate the Young's modulus $\log(E)$ by a piecewise constant field $\log(E^h)$ that has the same statistics as $\log(E)$ at the center of the elements.
Denoting by 
$\dimension=2197$ 
the number of elements in the mesh, we define $f:\mathbb{R}^{\dimension} \rightarrow \mathbb{R}$ as the map from the log--Young's modulus to the quantity of interest
$$ 
 f: X=\log(E^{h})  \mapsto u_{2}^{h}(\mathrm{PoI}).
$$

\begin{figure}[h]
  \centering 
  \begin{subfigure}[t]{0.48\textwidth}
  \centering 
    \includegraphics[width = \textwidth]{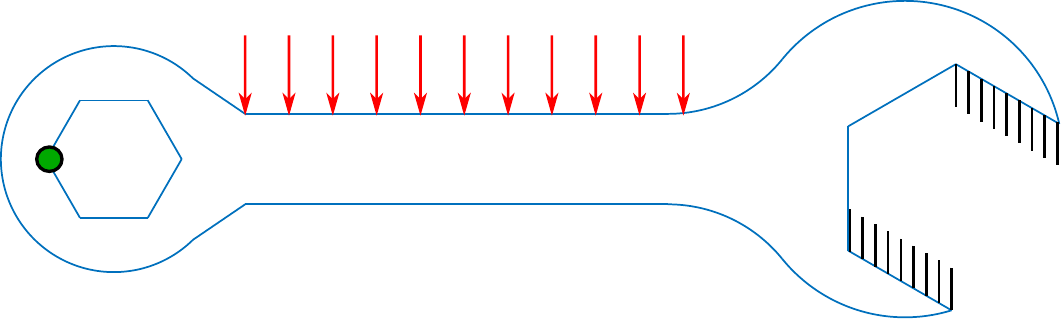} 
    \caption{Geometry and boundary conditions.}\label{fig:Wrench_GEOMETRY_1}
  \end{subfigure}
  \begin{subfigure}[t]{0.48\textwidth}
  \centering
    \includegraphics[width = \textwidth]{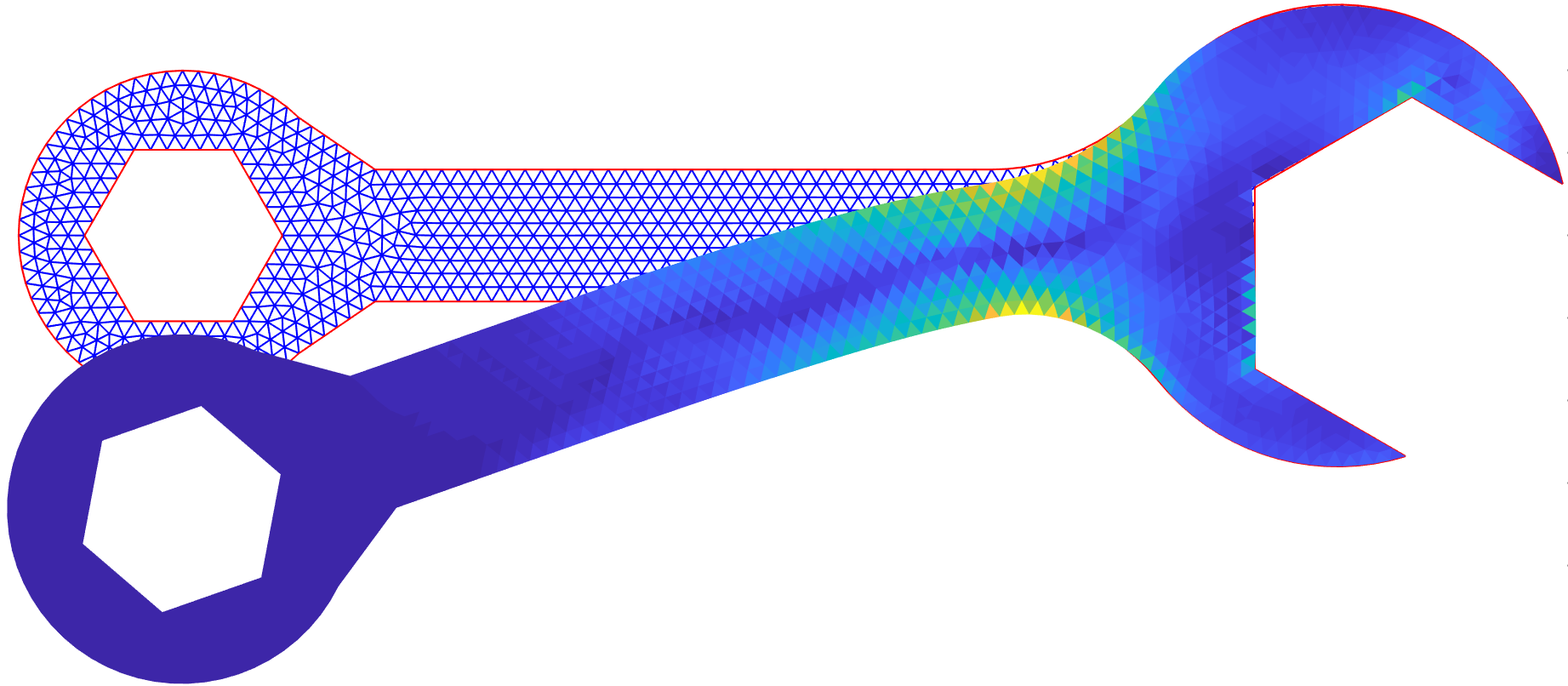}
   \caption{Mesh and realization of $u^h(X)$}\label{fig:Wrench_GEOMETRY_2}
  \end{subfigure}
  \caption{
   Left panel: Dirichlet condition $u=0$ (black chopped lines), vertical unitary linear forcing (red arrows), and point of interest $\mathrm{PoI}$ (green dot).
   Right panel: mesh and finite element solution $u^h(X)$ associated with one realization of $X$, the random Young's modulus. The color represents the von Mises stress.
  }
  \label{fig:Wrench_GEOMETRY}
\end{figure}

The gradients of $f$ are computed with the adjoint method; see for instance \cite{oberai2003solution}.
The complexity of computing the gradient increases with the number of elements in the mesh.
For this reason, we introduce a low-fidelity model $g$ that relies on the coarser mesh depicted in Figure~\ref{fig:Wrench_COARSE_FINE_SOLcoarse}. 
This coarse mesh contains 
$423$ 
elements, so that the function $g$ and its gradient can be evaluated with less computational effort compared to $f$. 
We now explain in detail how $g$ is defined. 
First, the log of the Young's modulus on the coarse mesh is defined as the piecewise constant field whose value at a given (coarse) element equals the spatial mean of $X=\log(E^h)$ restricted to that coarse element; see the illustration from Figure~\ref{fig:Wrench_COARSE_FINE_Efine} to Figure~\ref{fig:Wrench_COARSE_FINE_Ecoarse}.
Then, we compute the coarse finite element solution and we define $g(X)$ as the vertical displacement of the PoI.
Figure~\ref{fig:Wrench_COARSE_FINE_SOLfine} and Figure~\ref{fig:Wrench_COARSE_FINE_SOLcoarse} show that the fine and coarse solutions are similar.
With our implementation\footnote{Code available at \url{https://gitlab.inria.fr/ozahm/wrenchmark.git}.}, evaluating the gradient of the low-fidelity model $\gradapprox$ is approximately $7$ times faster than computing $\grad$.

\begin{figure}[h]
  \centering 
  \begin{subfigure}[t]{0.48\textwidth}
  \centering 
    \includegraphics[width = \textwidth]{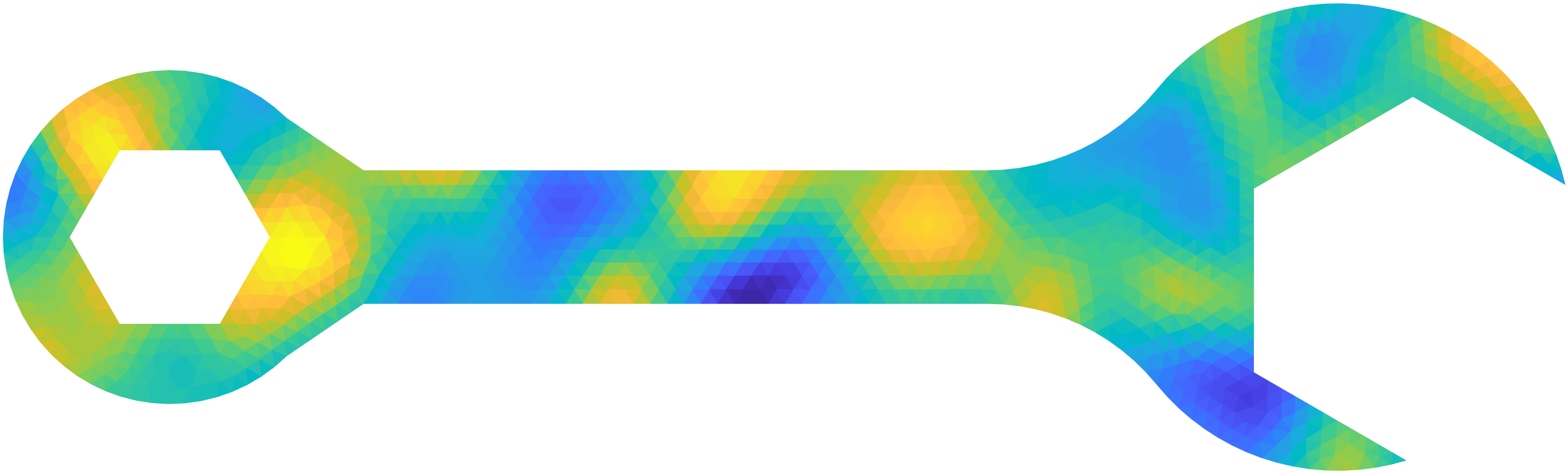} 
    \caption{Realization of $X=\log(E^h)$ on the fine mesh}\label{fig:Wrench_COARSE_FINE_Efine}
  \end{subfigure}
  ~
  \begin{subfigure}[t]{0.48\textwidth}
  \centering
    \includegraphics[width = \textwidth]{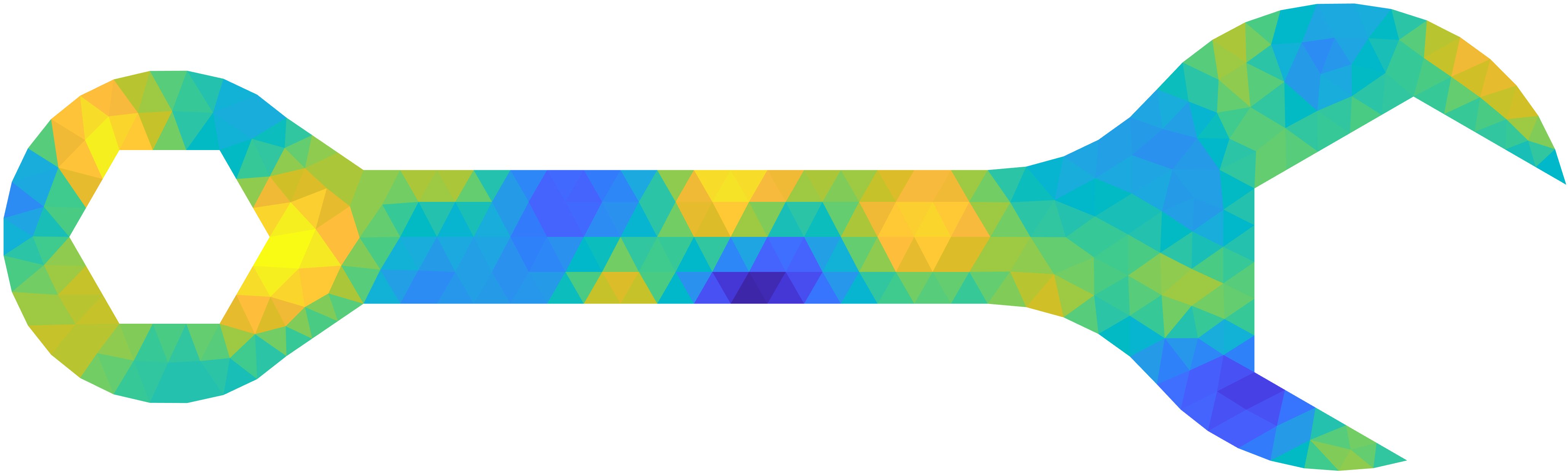}
   \caption{Projection of $X$ on the coarse mesh}\label{fig:Wrench_COARSE_FINE_Ecoarse}
  \end{subfigure}
  \begin{subfigure}[t]{0.48\textwidth}
  \centering 
    \includegraphics[width = \textwidth]{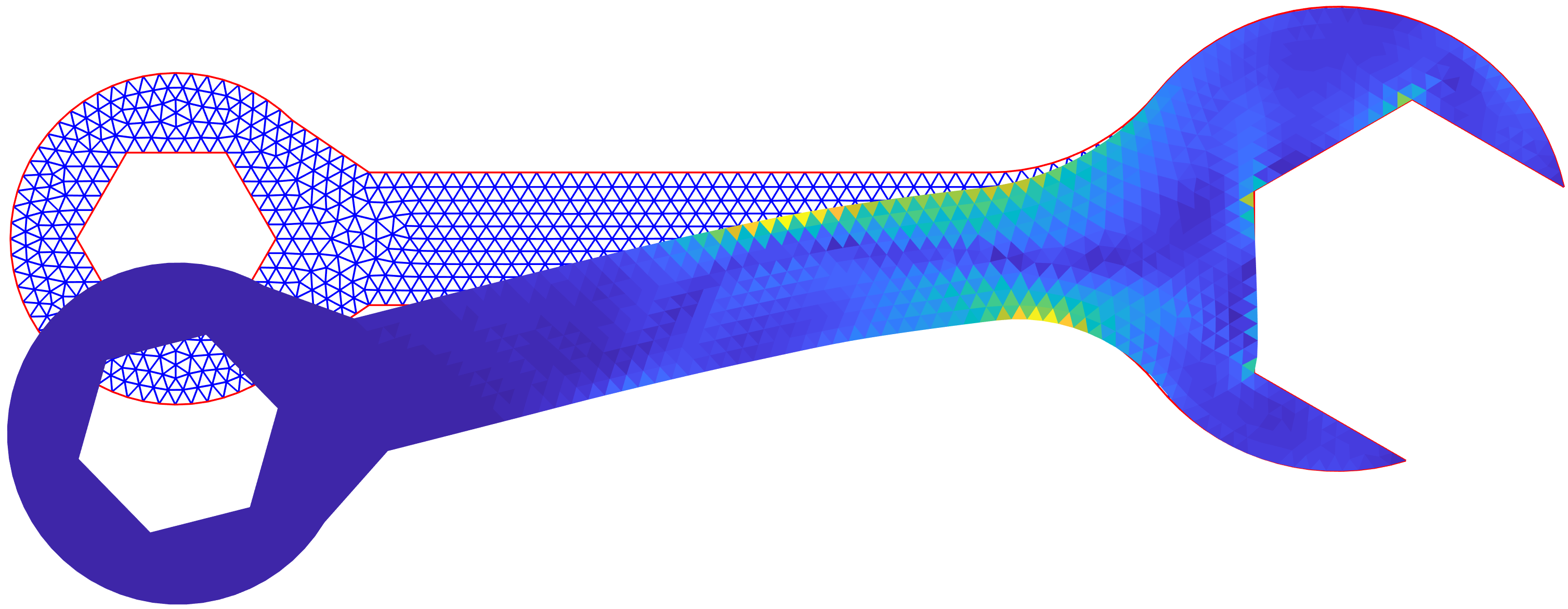} 
    \caption{Finite element solution on the fine mesh}\label{fig:Wrench_COARSE_FINE_SOLfine}
  \end{subfigure}
  ~
  \begin{subfigure}[t]{0.48\textwidth}
  \centering
    \includegraphics[width = \textwidth]{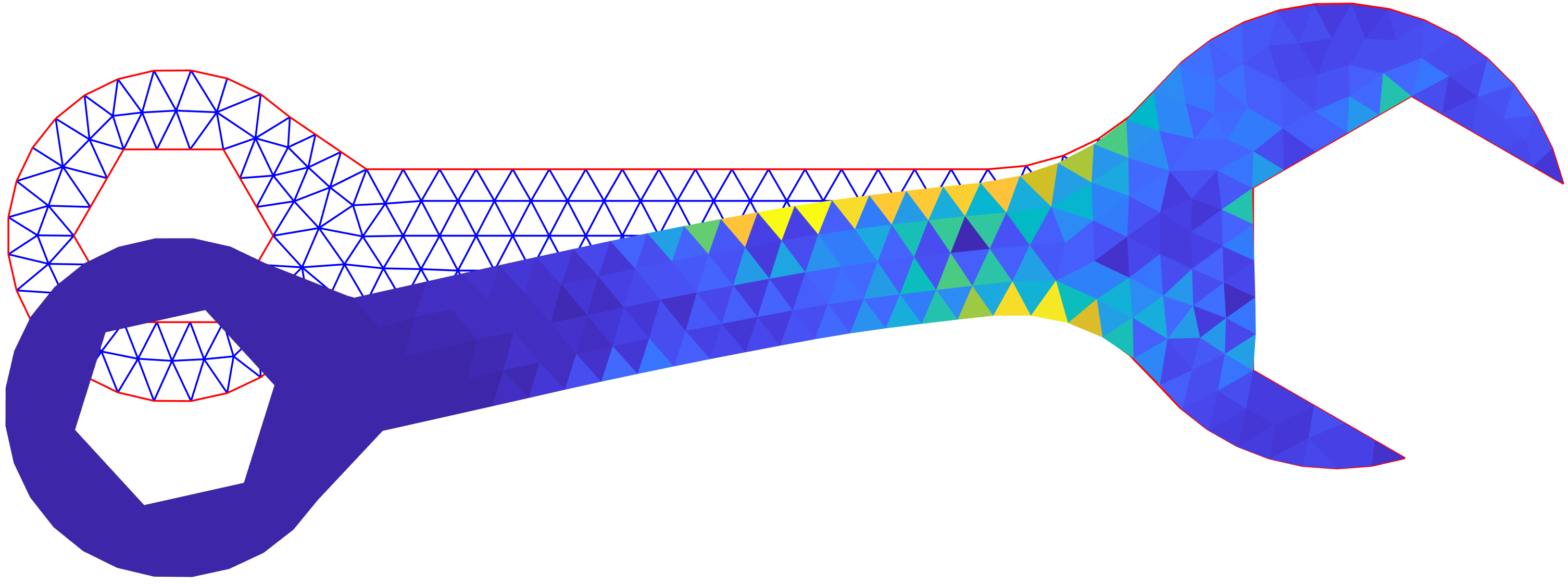}
   \caption{Finite element solution on the coarse mesh}\label{fig:Wrench_COARSE_FINE_SOLcoarse}
  \end{subfigure}
  \caption{
   Construction of the low-fidelity approximation $g$ of $f$ using a coarse mesh.
   Figure~\ref{fig:Wrench_COARSE_FINE_Efine} represents a realization of the Young's modulus on the fine mesh and Figure~\ref{fig:Wrench_COARSE_FINE_Ecoarse} its projection onto the coarser mesh.
   The corresponding finite element solution on the fine mesh (resp.\ coarse) and von Mises stress are represented in Figure~\ref{fig:Wrench_COARSE_FINE_SOLfine} (resp.\ Figure~\ref{fig:Wrench_COARSE_FINE_SOLcoarse}). 
  }
  \label{fig:Wrench_COARSE_FINE}
\end{figure}

Since the error bound \eqref{eq:AScontrol} holds when $X$ is a standard Gaussian random vector, we employ the following change of variables.
Denoting by $\Sigma$ the covariance matrix of $X\sim\mathcal{N}(0,\Sigma)$, we write $X = \Sigma^{1/2} X_\text{std} $
where $X_\text{std}\sim\mathcal{N}(0,I_d)$ and where $\Sigma^{1/2}$ is a positive square root of $\Sigma$.
After this change of variables, the high-fidelity and low-fidelity models are respectively $f_\text{std}: X_\text{std} \mapsto f(\Sigma^{1/2} X_\text{std})$ and $g_\text{std}: X_\text{std} \mapsto g(\Sigma^{1/2} X_\text{std})$ and the standardized AS matrix is $H_\text{std} = \expect[\nabla f_\text{std}(X_\text{std})\nabla f_\text{std}(X_\text{std})^T]$.
This change of variable can be interpreted as a preconditioning step of the AS matrix $H=\expect[\grad(X)\grad(X)^T]$ since we can write $H_\text{std} = \Sigma^{T/2}H\Sigma^{1/2}$.
For the sake of simplicity, we omit the subscript `std' and we use the notations $H$, $f$, and $g$  for the standardized version of the AS matrix, the high-fidelity model, and low-fidelity model.

\subsubsection{Numerical results}
We compute $10^{4}$ evaluations of the high-fidelity gradient $\grad$ and use them to form a SF estimator $\widehat{\matrixH}_{SF}^{\text{ref}}$ that we consider as a reference AS matrix.
We compare the performance of the MF and SF estimators computed on ten cases characterized by different computational budgets.
For a given case $\gamma\in\{1,\dots,10\}$, the SF estimator 
is computed with $m_{1}= 3 \gamma$ gradient evaluations while the MF estimator uses $m_{1}= 2 \gamma$ and $m_{2}= 5 \gamma$.
Given that evaluating $\gradapprox$ is 7 times faster than evaluating $\grad$, the computational costs of the SF and MF estimators are equivalent.
In the following, we refer to $\gamma$ as the cost coefficient, as we have set the computational budget to increase linearly with $\gamma$.

Figure~\ref{fig:wrench_error_spectrum} (left panel) shows the relative error with respect to the reference estimator $\widehat{\matrixH}_{SF}^{\text{ref}}$ as a function of the cost coefficient $\gamma$, averaged over 100 independent experiments.
The MF estimator outperforms the SF estimator for all the budgets tested.
In particular, the MF estimator (respectively SF estimator) reaches a relative error of 0.4 for $\gamma = 4$ (respectively $\gamma = 7$).
This represents a $42.8\%$ reduction in computational resources for the MF estimator.
Figure~\ref{fig:wrench_error_spectrum} (right panel) shows the eigenvalues of the AS matrix estimators, compared to those of the reference, for cost coefficient $\gamma=10$.
The MF estimator provides a better approximation of the spectral decay of the AS matrix than the SF estimator.
We note that the spectral decay suggests that a few modes ($\approx 5$) are sufficient to describe the behavior of the function $\obj$.
Finally, Figure~\ref{fig:wrench_modes} shows the two leading modes (eigenvectors) from the reference, the SF, and the MF estimators for cost coefficient $\gamma=10$.
We note that both the SF and the MF estimators recover the leading modes correctly.

\begin{figure}[h]
  \centering
  \begin{subfigure}[t]{0.48\textwidth}
  \centering
    \includegraphics[width = 0.9\textwidth]{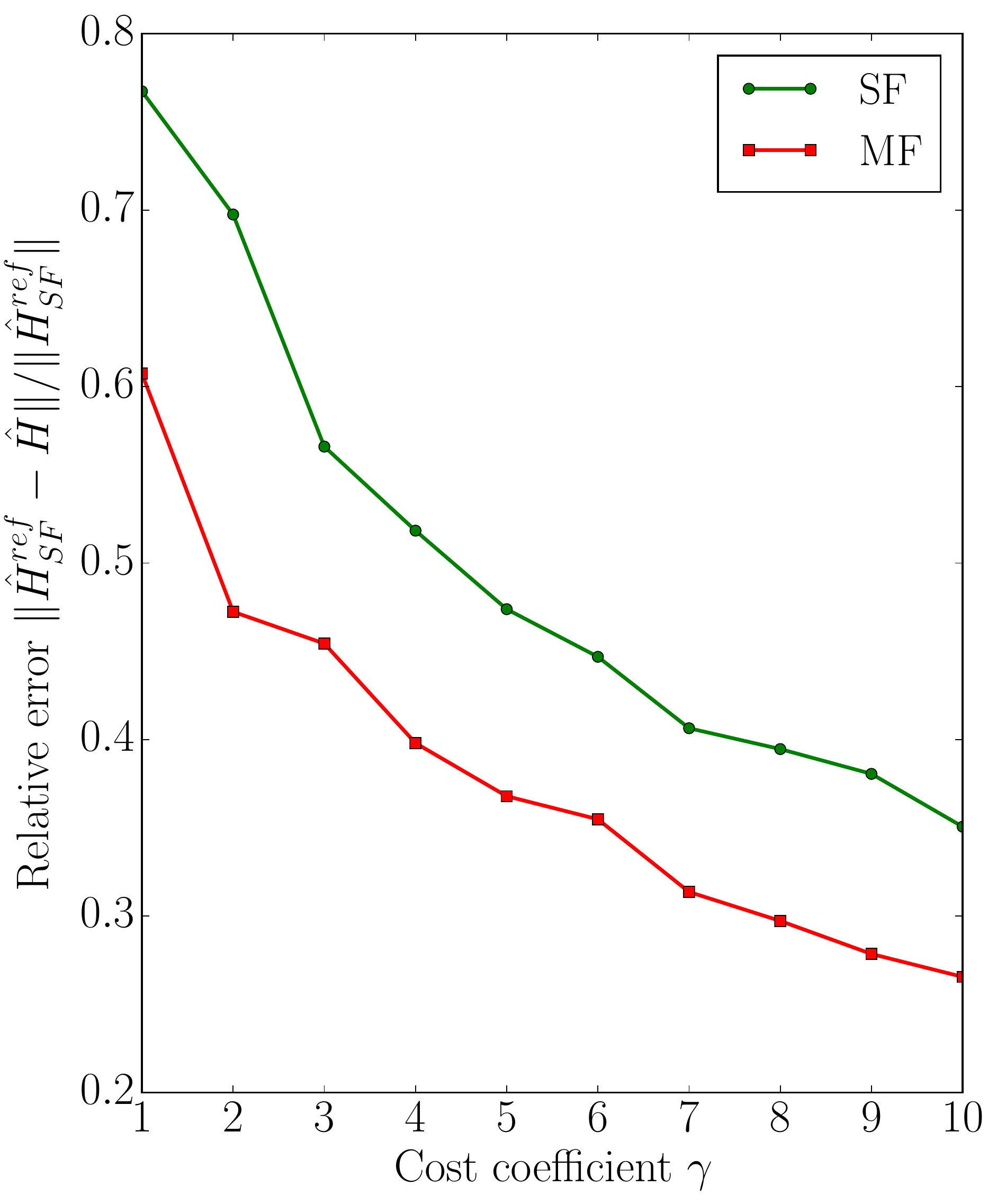}
  \end{subfigure}
  \begin{subfigure}[t]{0.48\textwidth}
  \centering
    \includegraphics[width = 0.9\textwidth]{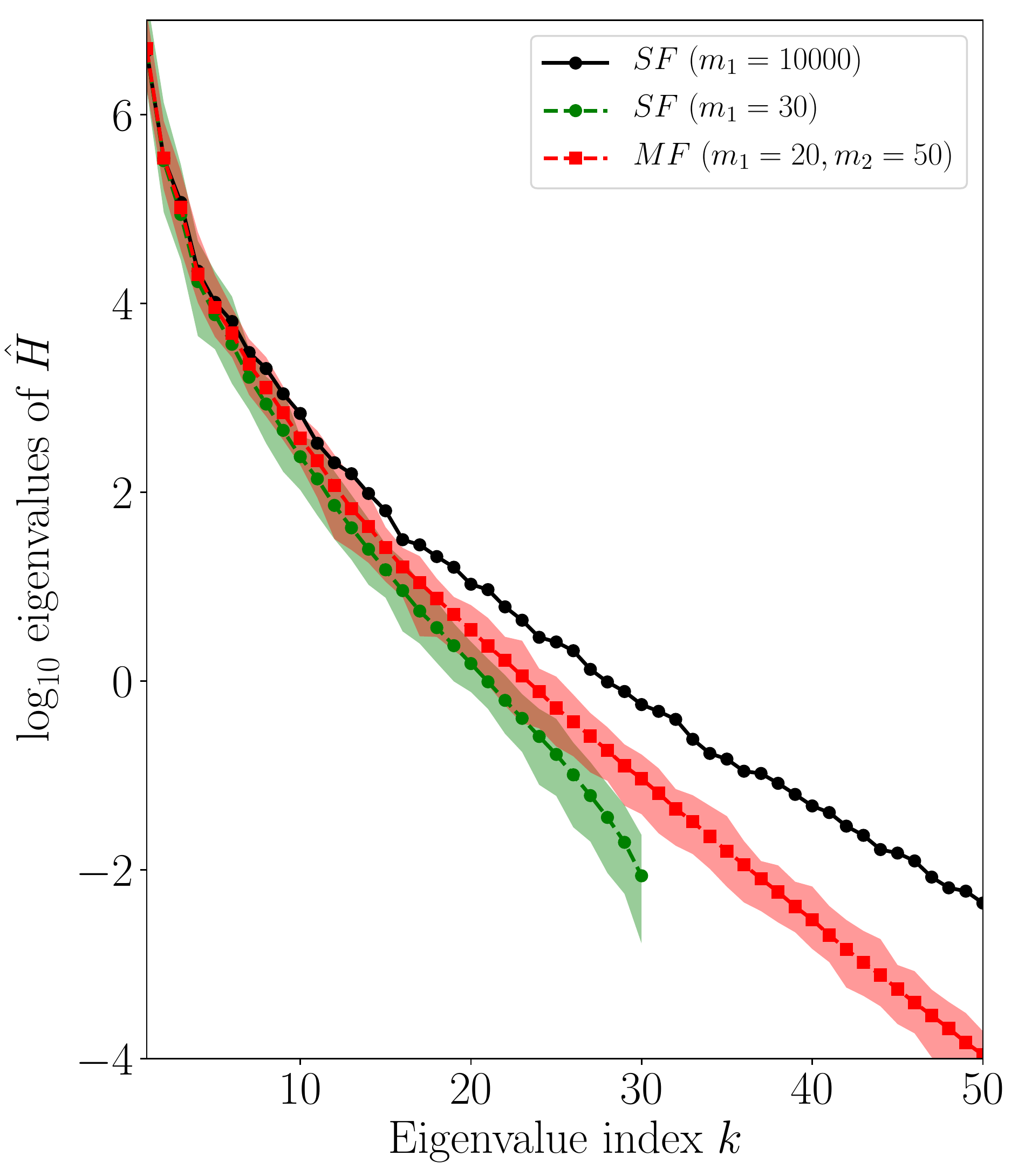}
    \end{subfigure}
  \caption{
  Left panel: Relative error as a function of cost coefficient $\gamma$, averaged over 100 independent experiments.
  At equivalent computational budget, the MF estimator outperforms the SF estimator.
  Right panel: Eigenvalues of the AS matrix estimators for the vertical displacement of the wrench PoI averaged over 100 independent experiments.
  Shadings represent the minimum and maximum values over the 100 independent experiments.
  The high-fidelity SF estimator used $m_{1}=30$ samples and
  the MF estimator is constructed with with $m_{1}=20$ and $m_{2}=50$ samples.
  This corresponds to a similar computational budget.
  }
  \label{fig:wrench_error_spectrum}
\end{figure}

\begin{figure}[h]
  \centering
    \includegraphics[width = 1.0\textwidth]{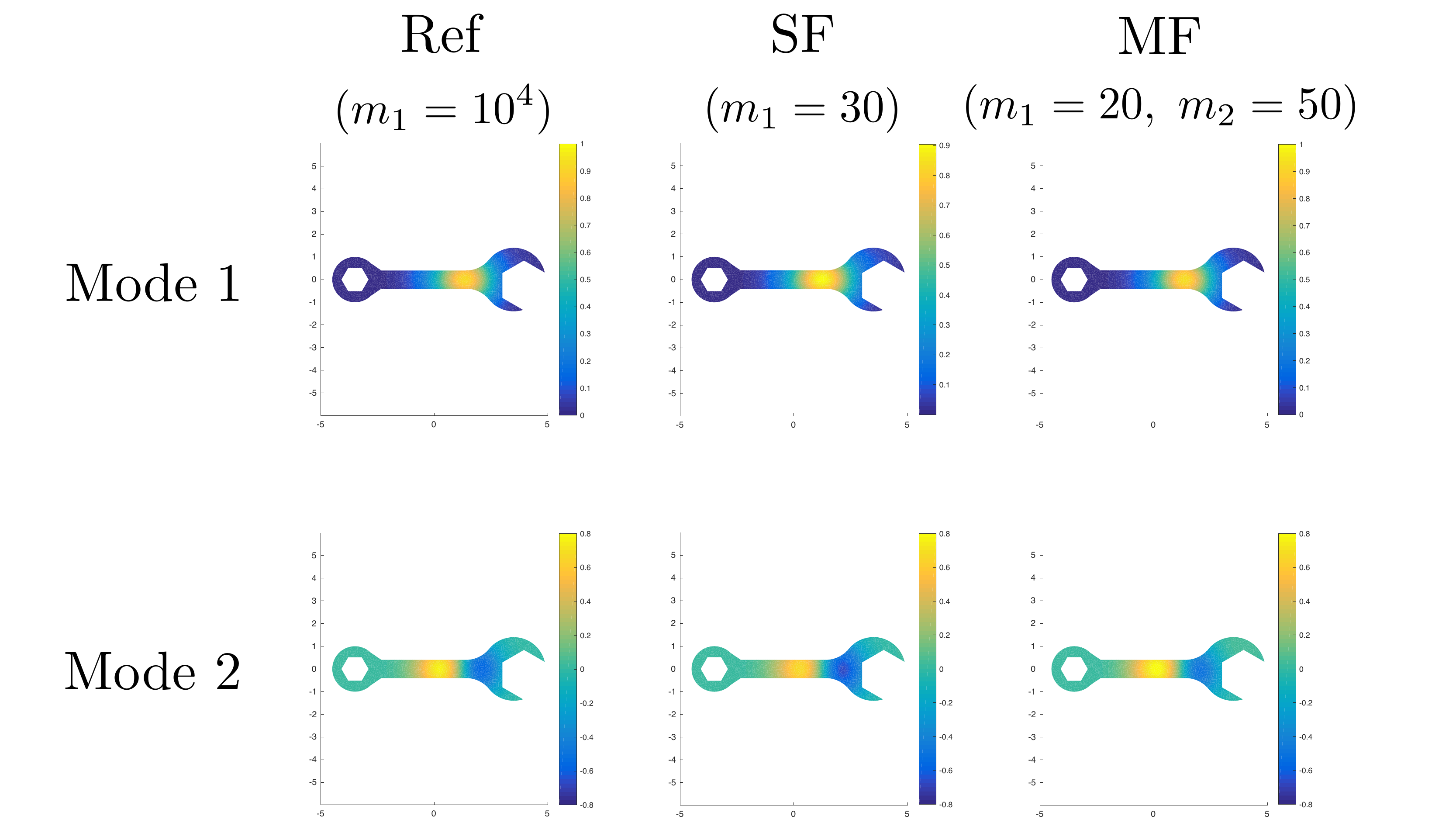}
    \caption{First  (top row) and second leading modes (bottom row) of the AS matrix for the reference estimator (left column), SF estimator (middle column) and MF estimator (right column) for the cost coefficient $\gamma=10$. Both the SF and the MF estimators correctly capture the two leading modes of the reference AS matrix.}
  \label{fig:wrench_modes}
\end{figure}

\subsection{ONERA M6 wing shape optimization} 
\label{sub:shape_optimization}

In this section, we illustrate the proposed MF estimator on an expensive-to-evaluate engineering example.

We consider the problem of finding the active subspace associated with the shape optimization of the ONERA M6 wing.
The shape of the wing is parameterized by free form deformation (FFD) box control points.
In our experiment we used $5\times8\times2=80$ FFD boxes.
Imposing second-order continuity of surfaces with the FFD fixes $5\times3\times2 =30$ variables, leaving $\dimension=50$ control points.
These correspond to the input variable $\design \in \inputspace=[-0.05,0.05]^{50}$.
The functions of interest are the drag and the lift coefficients.
Those quantities are computed using expensive-to-evaluate computational fluid dynamics (CFD) tools.
We use the SU2\footnote{\url{https://github.com/su2code/SU2/tree/ec551e427f20373511432e6cd87402304cc46baa}} package \cite{palacios2013stanford} to solve the  Reynolds-averaged Navier-Stokes (RANS) equations and compute the gradients using the continuous adjoint method.
The flow conditions are such that the Mach number is $M_\infty=0.8395$, the angle of attack of the wing is $\alpha = \ang{3.03}$, and the Reynolds number is $Re =11.72\times10^{6}$.
We use the Spalart-Allmaras turbulence model.

The high-fidelity function uses the aforementioned CFD model with the following stopping criteria:
the solution of the RANS equations and the associated adjoint are computed with a limit of $10^{4}$ iterations or fewer  if the Cauchy convergence criteria reaches $10^{-5}$ within this limit (i.e., maximum variation of the quantity of interest over 100 iterations is lower than $10^{-5}$).
For the low-fidelity function $\approxfunc$, we use the same model as the high-fidelity function $\obj$ but reduce the maximum number of iterations allowed for convergence from $10^{4}$ to $10^{3}$.
An evaluation of the high-fidelity model (drag and lift coefficients and associated gradients) is thus approximately $10$ times more expensive than the low-fidelity model.

We compute a total of $100$ high-fidelity evaluations and $500$ low-fidelity evaluations.
We use the 100 evaluations of the high-fidelity model to compute a first ``reference'' SF estimator ($m_{1}=100$) that we denote $\HSF^{(100)}$.
We split the collected data into 5 independent experimental batches and construct three estimators per batch.
The first is a SF estimator constructed with $m_{1}=10$ high-fidelity evaluations.
The second is a SF estimator built with $m_{1}=20$ high-fidelity evaluations.
The last estimator is a MF estimator with $m_{1}=10$ and $m_{2}=90$ evaluations.
For these three estimators, the $m_{1}$ high-fidelity evaluations are common.
We note that the computational cost of the second SF estimator ($m_{1}=20$) and the MF estimator ($m_{1}=10$ and $m_{2}=90$) are similar.

\begin{figure}[h]
  \centering
  \begin{subfigure}[t]{0.48\textwidth}
  \centering
    \includegraphics[width = 0.9\textwidth]{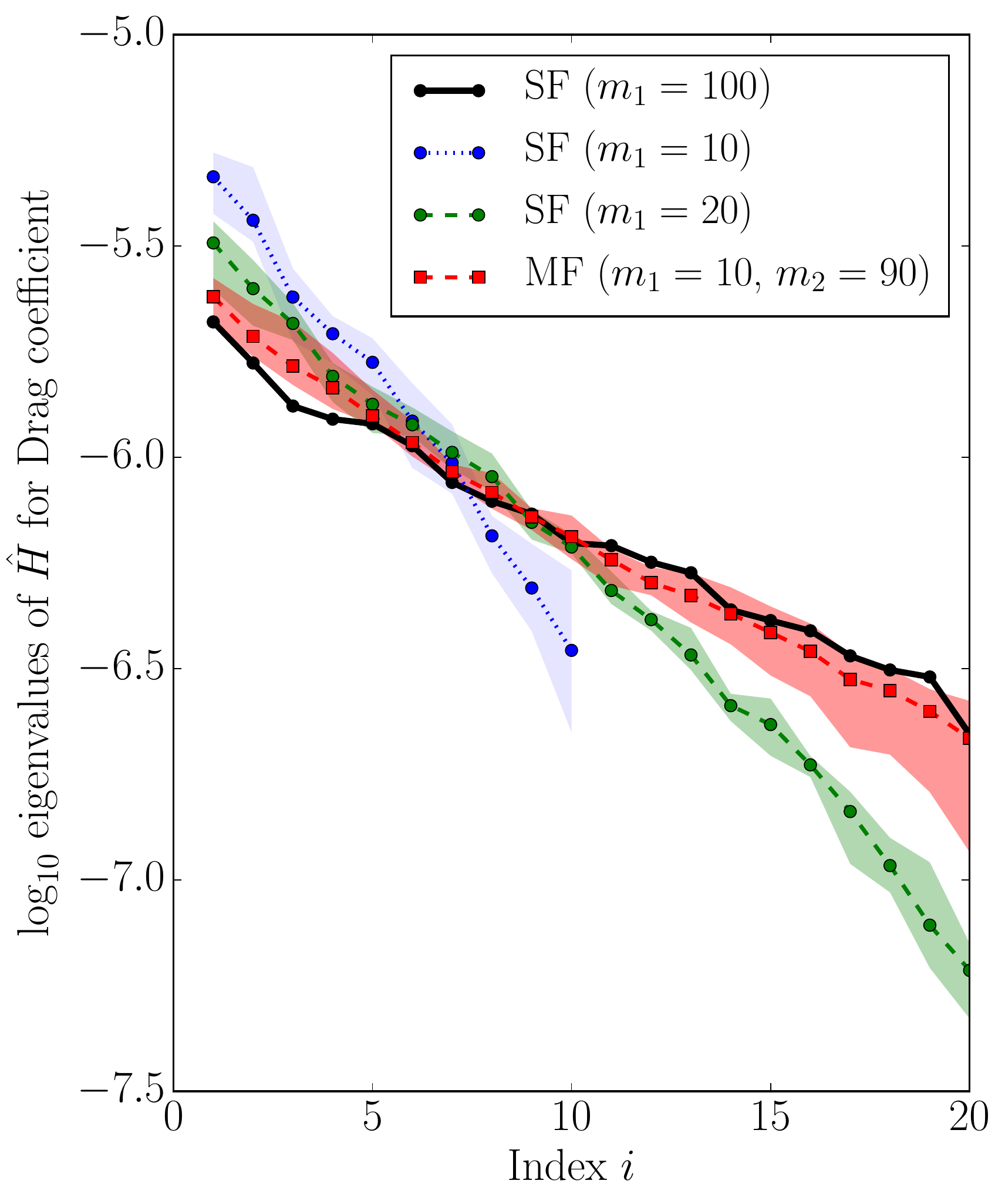}
  \end{subfigure}
  \begin{subfigure}[t]{0.48\textwidth}
  \centering
    \includegraphics[width = 0.9\textwidth]{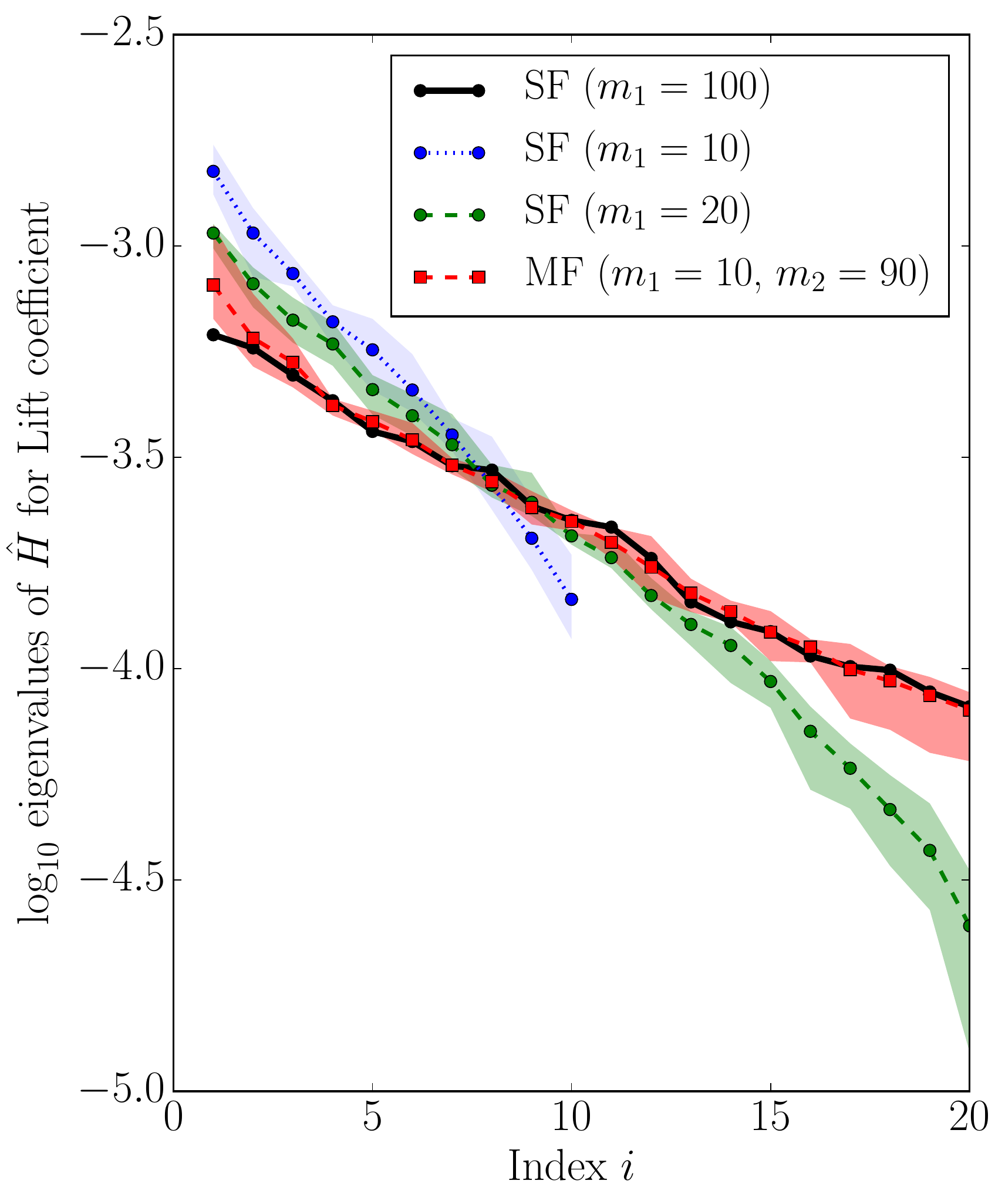}
    \end{subfigure}
  \caption{Eigenvalues of the AS matrix estimators for the drag coefficient (left panel) and for the lift coefficient (right panel) averaged over 5 independent experiments.
  The SF estimator with $m_{1}=20$ and the MF estimator (dashed lines) used the same computational budget to evaluate the gradients.
  At equal budget, the MF estimator provides a better estimate of the spectrum of $\HSF^{(100)}$ than the SF estimator.
  Shadings represent maximum and minimum values over the 5 independent experiments.
  }
  \label{fig:ONERA_comparison}
\end{figure}
\begin{figure}[h]
  \centering
  \begin{subfigure}[t]{0.48\textwidth}
  \centering
    \includegraphics[width = 0.9\textwidth]{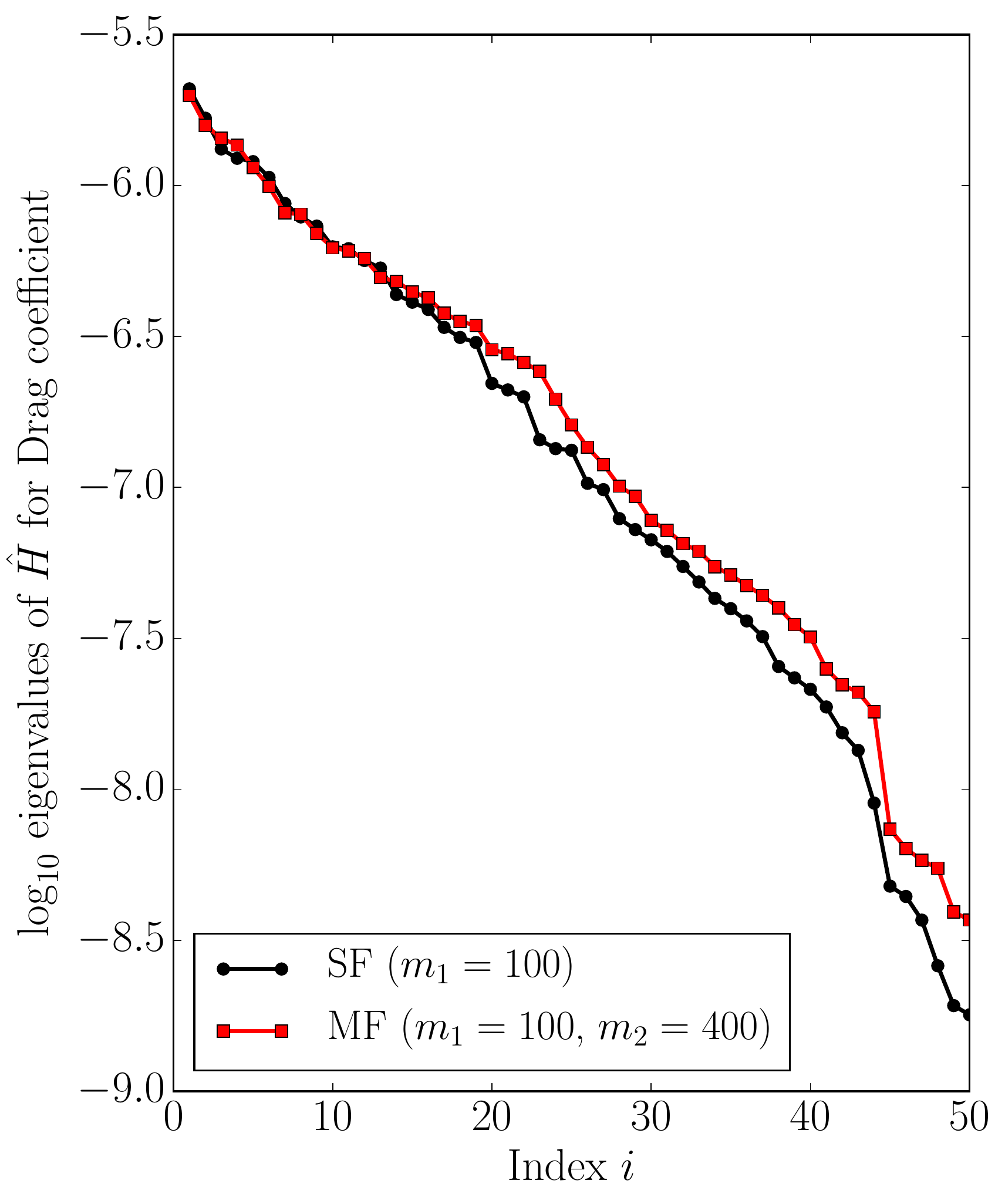}
  \end{subfigure}
  \begin{subfigure}[t]{0.48\textwidth}
  \centering
    \includegraphics[width = 0.9\textwidth]{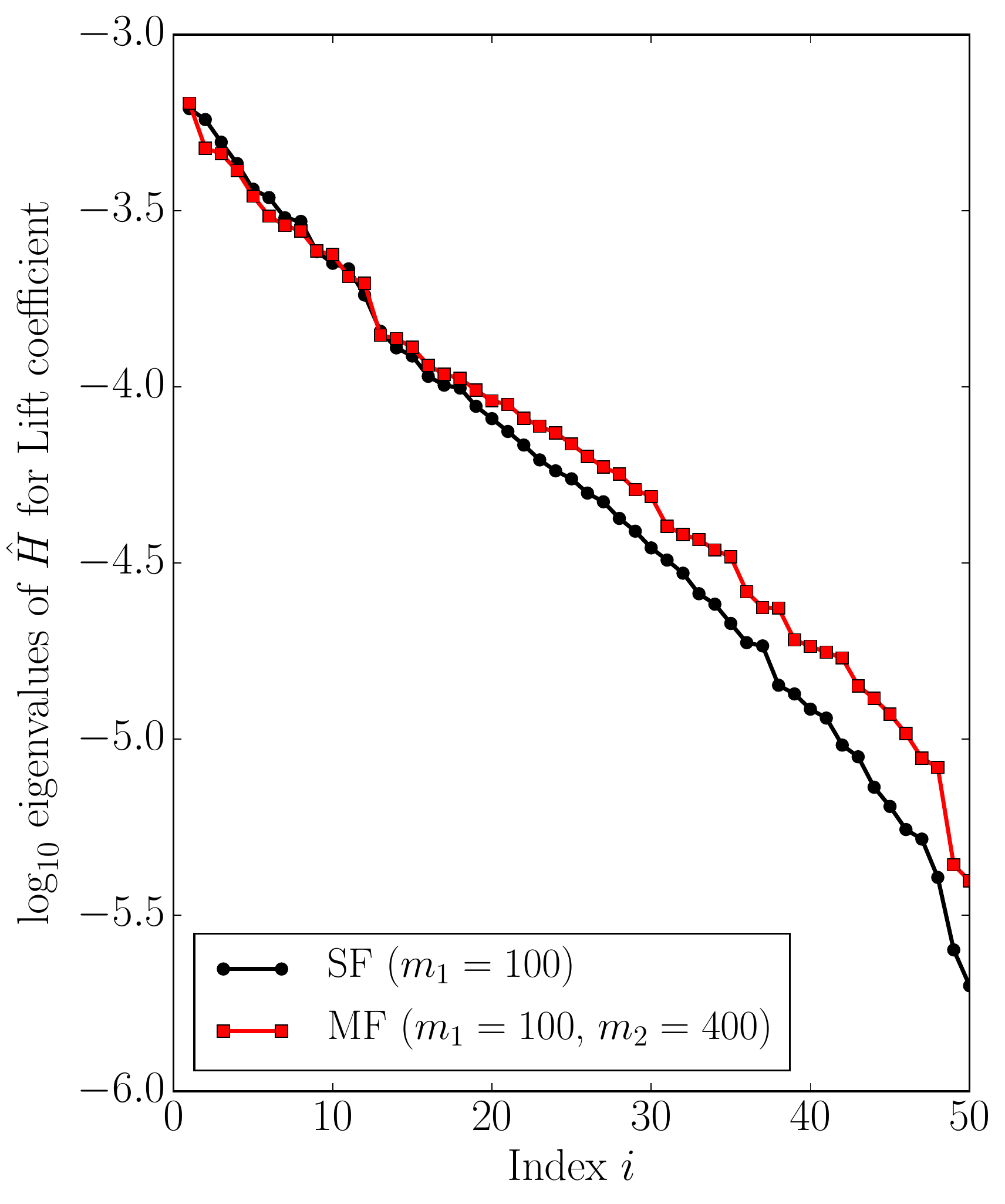}
    \end{subfigure}
  \caption{Eigenvalues of the AS matrix estimators for the drag coefficient (left panel) and for the lift coefficient (right panel).
  The high-fidelity SF estimator used $m_{1}=100$ samples and
  the MF estimator is constructed with with $m_{1}=100$ and $m_{2}=400$ samples.}
  \label{fig:ONERA_reference}
\end{figure}

Figure~\ref{fig:ONERA_comparison} shows the 20 first eigenvalues of the three estimators (averaged over 5 independent batches) and $\HSF^{(100)}$.
Note that the SF estimators (except $\HSF^{(100)}$) are rank deficient (rank 10 and rank 20).
The MF estimator is full rank (rank 50).
The MF estimator is closer to $\HSF^{(100)}$ than the two SF estimators.
In particular, the MF estimator outperforms the second SF estimator with similar computational budget.
The error bars show the maximum and the minimum values over the 5 independent experiments for each estimator and confirm the robustness of the proposed method.

Figure~\ref{fig:ONERA_reference} shows the 50 eigenvalues of the SF estimator $\HSF^{(100)}$ ($m_{1}=100$) and a MF estimator using all the available evaluations ($m_{1}=100$, $m_{2}=400$).
The leading eigenvalues are similar for both estimators.
The difference between the two estimators increases for lower eigenvalues (higher indices).
For both estimators, we compute approximations of the characteristic quantities $\intdim$, $\normE$, $\normH$, $\beta^{2}$, and $\theta^{2}$ and summarize them in Table~\ref{tab:estimator_characteristics}.
The SF quantities are computed based on the $m_{1}=100$ high-fidelity evaluations: 
$\normH$ and $\intdim$ are computed using $\widehat{\matrixH}_{SF}^{(100)}$ in lieu of $\matrixH$, 
$\normE$ is approximated by $\frac{1}{m_{1}}\sum_{i=1}^{m_{1}}\Vert\grad(X_{i})\Vert^{2}$, 
while $\beta^{2}$ is approximated by 
$\max_{i}\Vert\grad(X_{i})\Vert^{2}/\normE$ for $i\in\{1, \dots,m_{1}\}$.
The MF quantities are computed based on the $m_{1}=100$ high-fidelity evaluations and $m_{1}+m_{2}=500$ low-fidelity evaluations:
$\normH$ and $\intdim$ are computed using the MF estimator in lieu of $\matrixH$,
$\normE$ is approximated by the  (scalar) MF estimator $\frac{1}{m_{1}}\sum_{i=1}^{m_{1}}(\Vert\grad(X_{i})\Vert^{2}-\Vert\gradapprox(X_{i})\Vert^{2})+\frac{1}{m_{2}}\sum_{i=m_{1}+1}^{m_{1}+m_{2}}\Vert\gradapprox(X_{i})\Vert^{2}$, 
while  $\theta^{2}$ is approximated by 
$\max_{i}\Vert\grad(X_{i})-\gradapprox(X_{i})\Vert^{2}/\normE$ for $i\in\{1, \dots,m_{1}\}$.
\begin{table}[h]
      \caption{Approximation of the characteristic quantities for the SF and MF estimators for the drag and lift coefficients.}
      \begin{minipage}{\textwidth}
      \centering
        \vspace{5pt}
        \setlength{\tabcolsep}{10pt}
        \begin{tabular}{ccccc}
        \toprule
        &\multicolumn{2}{c}{Drag coefficient $C_{d}$}&\multicolumn{2}{c}{Lift coefficient $C_{l}$} \\
        \cmidrule(r){2-3}\cmidrule(r){4-5}
            & SF & MF& SF & MF \\
        \midrule
        $\normE$       &   $1.74\times 10^{-5}$   &   $1.81\times 10^{-5}$     &   $5.96\times 10^{-3}$   &   $6.08\times 10^{-3}$  \\
        $\normH$       &   $2.09\times 10^{-6}$   &    $1.99\times 10^{-6}$    &   $6.16\times 10^{-4}$   &   $6.38\times 10^{-4}$  \\
        $\intdim$      &   8.34   &    9.10    &   9.67   &  9.53   \\
        $\beta^{2}$    &   1.76   &    -    &  1.54    &   -  \\
        $\theta^{2}$   &   -   &    $1.30\times 10^{-5}$    &   -   &   $1.03\times 10^{-4}$  \\
        \bottomrule
        \end{tabular} 
      \end{minipage}
       \label{tab:estimator_characteristics}
    \end{table}
From Table~\ref{tab:estimator_characteristics}, it can be seen that the intrinsic dimension of the AS matrix is approximately 9 for both the drag and lift coefficients.
Table~\ref{tab:estimator_characteristics} also shows that the parameter $\beta^{2}$ is not large ($\leq 2$), which indicates, along with the low intrinsic dimension, that computing a good estimator of $\matrixH$ requires relatively few evaluations.
We also note that the low-fidelity model is a good approximation of the high-fidelity model, with $\theta^{2}\leq1.5\times10^{-5}$ for the drag and $\theta^{2}\leq1.1\times10^{-4}$ for the lift.
Figure~\ref{fig:ONERA_energy} shows the normalized sum of the eigenvalues of the best rank-$r$ approximation of $\widehat{\matrixH}$.
Selecting the 20 first eigenvectors of $\widehat{\matrixH}$ allows us to capture more than 80\% of the spectral content and decreases by 30 the dimensionality of the shape optimization problem (a $60\%$ decrease).

\begin{figure}[h]
  \centering
    \includegraphics[width = 0.7\textwidth]{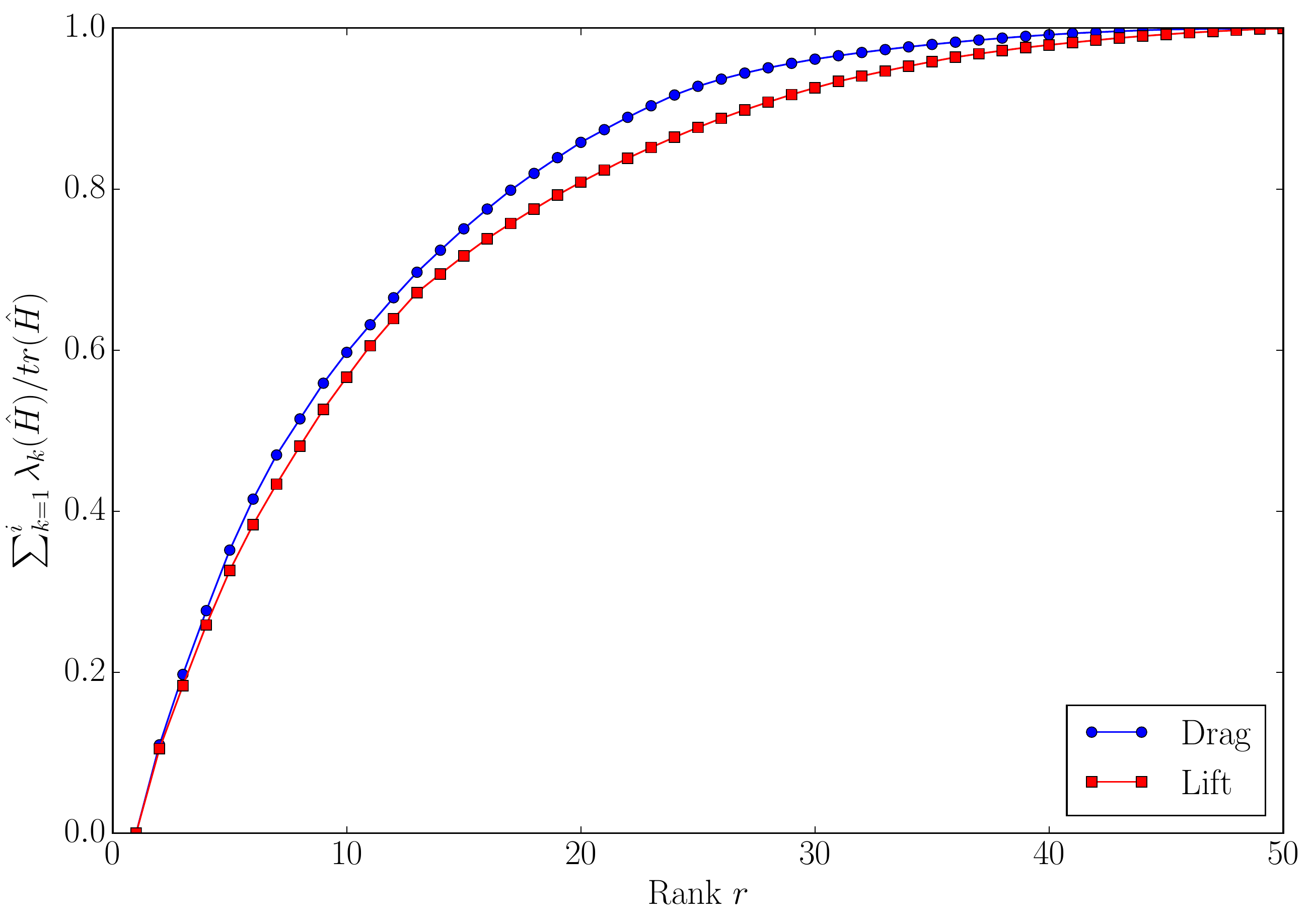}
    \caption{Percentage of the spectrum $\sum_{k=1}^{i}\lambda_{r}(\widehat{\matrixH})/ tr(\widehat{\matrixH})$ captured by the best rank-$r$ approximation of $\widehat{\matrixH}$, where $\lambda_{k}(\widehat{\matrixH})$ is the $k^{th}$ eigenvalue of $\widehat{\matrixH}$.}
  \label{fig:ONERA_energy}
\end{figure}

\section{Conclusions}
\label{sec:conclusions}
We proposed a multifidelity approach to identify low-dimensional subspaces capturing most of the variation of a function of interest.
Our approach builds on the gradient-based active subspace methodology, which seeks to compute the matrix $\matrixH$ containing the second moments of the gradient function.
The proposed approach reduces the computational cost of Monte Carlo methods used to estimate this matrix, by using a low-fidelity, cheap-to-evaluate gradient approximation as a control variate.
The performance improvements of the resulting multifidelity (MF) estimator $\HMF$ are demonstrated on two engineering examples governed by partial differential equations:
a high-dimensional linear elasticity problem defined over more than two thousand input variables and an expensive shape optimization problem defined over 50 input variables.

Analysis of the performance of the multifidelity technique yields error bounds for the matrix error $\Vert\matrixH-\HMF\Vert$ both in expectation and with high probability.
These error bounds depend on the intrinsic dimension of the problem, which is related to the spectral decay of the active subspace matrix.
When a function varies mostly along a few directions, the intrinsic dimension is low.
In such a case, approximating $\matrixH$ may require only a small number of gradient evaluations.
This relationship was confirmed empirically by a parametric study conducted on two analytical problems: lower intrinsic dimension led to lower matrix error for the same number of high-fidelity evaluations.

The performance improvements of the multifidelity approach are threefold.
First, we showed that the MF estimator reduces the cost of performing dimension reduction.
This was illustrated on the linear elasticity problem, where the multifidelity approach needed about $43\%$ less computational effort than its single-fidelity counterpart to achieve the same relative error.
Second, we showed that the multifidelity approach was able to recover eigenvectors qualitatively similar to the exact ones.
Third, the MF estimator led to better estimates of the spectral decay of $\matrixH$.
On the linear elasticity problem, which is characterized by a low intrinsic dimension ($\leq 5$), the multifidelity approach led to better estimates, especially for smaller eigenvalues.
On the shape optimization problem, which is characterized by a higher intrinsic dimension ($\approx 9$), the MF estimator led to better estimates for all eigenvalues.
This behavior is in contrast to the single-fidelity method, which overestimated the leading eigenvalues and underestimated the lower eigenvalues, thereby underestimating the intrinsic dimension of $\matrixH$ and overestimating the spectral decay.
Recovering the spectral decay of $\matrixH$ is particularly important since one popular way of choosing the dimension $r$ of the active subspace is based on the spectral gap between consecutive eigenvalues.
If the spectral decay is overestimated, $r$ might be chosen too small to correctly capture the behavior of the high-fidelity function.
By providing a better estimate of the spectral decay, the MF estimator reduces this risk.

The dimension of the active subspace can also be chosen to control an error bound on the associated function approximation error, i.e., when using the active subspace to construct a ridge approximation of the original function.  
We showed that the approximation of this bound, which depends on the unknown $\matrixH$, improves as the matrix error decreases.  Because the MF estimator reduces the error in estimating $\matrixH$ for a given computational effort, the proposed multifidelity approach leads to a better selection of the active subspace dimension $r$.

\appendix
\section{Proofs of main results}
\label{sec:appendices}
In this section, we present the details of the proofs of Proposition \ref{prop:MainResult} (Appendix \ref{sec:MainProof}) and Proposition \ref{prop:SingleFidelityResult} (Appendix \ref{sec:SingleFidelityResult}).

\subsection{Proof of Proposition \ref{prop:MainResult}}\label{sec:MainProof}

The proof of Proposition \ref{prop:MainResult} relies on the concentration inequality known as the matrix Bernstein theorem.
It corresponds to Theorem~6.1.1 from \cite{tropp_intro_MAL-048} restricted to the case of real symmetric matrices.

\begin{theorem}[Matrix Bernstein: real symmetric case \cite{tropp_intro_MAL-048}]
\label{th:bernsteinNONintrinsic}
  Let $S_1,\hdots,S_m$ be $m$ independent zero-mean random symmetric matrices in $\mathbb{R}^{\dimension\times\dimension}$.
  Assume there exists $L<\infty$ such that 
  \begin{align*}
    \Vert S_i \Vert \leq L ,
  \end{align*}
  almost surely for all $1\leq i \leq m$, and let $v$ be such that
  $$
   \| \expect[ (S_1+\hdots+S_m)^2] \| \leq v.
  $$
  Then,
  \begin{align*}
    \expect [\Vert S_1+\hdots+S_m \Vert] 
    & \leq  \sqrt{2 v \log(2\dimension) } + \frac{1}{3} L \log(2d) .
  \end{align*}
  Furthermore, for any $t\geq 0$ we have
  \begin{align*}
    \proba\{\Vert S_1+\hdots+S_m \Vert \geq t\} \leq 2\dimension \exp\left(\frac{-t^{2}/2}{v+Lt/3}\right) .
  \end{align*}
  
\end{theorem}

In order to apply the matrix Bernstein theorem, we need to express the difference between the AS matrix $\matrixH$ and the MF estimator $\HMF$ as the sum of independent matrices.
We write 
$
 \matrixH-\HMF = S_1 + \hdots + S_m ,
$
where $m=m_1+m_2$ and 
$$
  S_i = 
  \left\{\begin{array}{ll} 
   \frac{1}{m_1} \big( \matrixH-\matrixG-(\grad(X)\grad(X)^T-\gradapprox(X)\gradapprox(X)^T) \big),
   & \text{if } 1\leq i\leq m_1 ,\\ 
   \frac{1}{m_2}\big(\matrixG-\gradapprox(X)\gradapprox(X)^T \big).
   & \text{otherwise},
  \end{array} 
  \right.  
$$
where $G = \expect[\gradapprox(X)\gradapprox(X)^{T}]$.
The following property provides bounds for $\|S_i\|$ and $\expect [\Vert S_1+\hdots+S_m \Vert]$.
Those bounds are expressed as functions of $m_{1}$, $m_{2}$, $\beta$, $\theta$, and $\normE$.

\begin{property}
  \label{prop:MFcharacteristics}
  Assumptions \eqref{eq:Assumption BETA} and \eqref{eq:Assumption THETA} yield
  \begin{equation}\label{eq:MFcharacteristics_VARIANCE}
    \| \expect[ (S_1+\hdots+S_m)^2] \|
    \leq \left( \frac{\theta^{2} (2+\theta)^{2}}{m_1}  + \frac{(\theta+\beta)^{2}(1 + \theta)^{2}}{m_2}\right)\normE^{2},
  \end{equation}
  and
  \begin{equation}\label{eq:MFcharacteristics_BOUND}
    \|S_i\| \leq \operatorname{max}\left\{\frac{2\theta(2\beta+\theta)}{m_1};\frac{2(\theta+\beta)^{2}}{m_2}\right\}\normE .
  \end{equation}
  almost surely for all $1\leq i\leq m$.
\end{property}

\begin{proof}
  We first derive the bound \eqref{eq:MFcharacteristics_VARIANCE} for the variance of the estimator before proving the bound \eqref{eq:MFcharacteristics_BOUND} for the norm of the summands.

  Using the independence of the summands $S_i$ and the fact that $\expect[S_i]=0$, we have
 \begin{align}
   \expect[ (S_1+\hdots+S_m)^2]  
  &=  \expect[S_1^2] + \hdots + \expect[S_{m_1}^2] + \expect[S_{m_1+1}^2] + \hdots + \expect[S_{m}^2] \nonumber \\
  &=  \frac{1}{m_1} \expect[(A-\expect[A])^2] + \frac{1}{m_2}\expect[(B-\expect[B])^2] , \label{eq:tmp5329581}
 \end{align}
 where
 \begin{align*}
  A&=\grad(X)\grad(X)^T-\gradapprox(X)\gradapprox(X)^T ,\\
  B&=\gradapprox(X)\gradapprox(X)^T .
 \end{align*}
 Notice that $0\mleq \expect[(A-\expect[A])^2] = \expect[A^2]-\expect[A]^2 \mleq \expect[A^2]$ so that $\|\expect[(A-\expect[A])^2]\|\leq\|\expect[A^2]\|$, where $\mleq$ denotes the Loewner partial order.
 Similarly, one has $\|\expect[(B-\expect[B])^2]\|\leq\|\expect[B^2]\|$. 
 Taking the norm of \eqref{eq:tmp5329581} and using a triangle inequality yields
 $$
  \|\expect[ (S_1+\hdots+S_m)^2]  \| \leq \frac{1}{m_1} \|\expect[A^2]\| + \frac{1}{m_2} \|\expect[B^2]\| .
 $$
 To obtain \eqref{eq:MFcharacteristics_VARIANCE}, it remains to show (i) that $\|\expect[A^2]\| \leq \theta^2 (2+\theta)^2 \normE^2$ and (ii) that $\|\expect[B^2]\| \leq (\beta+\theta)^2(1+\theta)^2\normE^2$.
 Let $u\in\mathbb{R}^{\dimension}$ such that $\|u\|\leq 1$.
 We have
 \begin{align*}
  u^T \expect[ A^2] u 
  &= \expect[ u^T ( \grad(X)\grad(X)^T-\gradapprox(X)\gradapprox(X)^T )^2 u]  \\
  &= \expect[ \|  ( \grad(X)\grad(X)^T-\gradapprox(X)\gradapprox(X)^T ) u \|^2 ]  \\
  &= \expect[ \|  \grad(X)( \grad(X) -\gradapprox(X) )^T u -(\gradapprox(X)-\grad(X)) \gradapprox(X)^T u \|^2 ]  .
 \end{align*}
 Using triangle inequalities, we can write
 \begin{align}
  u^T \expect[ A^2] u 
  &\leq \expect[ ( \|  \grad(X)( \grad(X) -\gradapprox(X) )^T u \|+\|(\gradapprox(X)-\grad(X)) \gradapprox(X)^T u \| )^2 ]  \nonumber\\
  &\leq \expect[ ( \|  \grad(X) \| \, \| \grad(X) -\gradapprox(X) \| +\|\gradapprox(X)-\grad(X)\| \, \| \gradapprox(X) \| )^2 ]  \nonumber\\
  &\leq \expect[ \| \grad(X) -\gradapprox(X) \|^2 (\|  \grad(X) \|+\| \gradapprox(X) \| )^2 ]  \nonumber\\
  &\leq \expect[ \| \grad(X) -\gradapprox(X) \|^2 (2\| \grad(X) \|+\| \grad(X)-\gradapprox(X) \| )^2 ]  \nonumber\\
  &\leq \theta^2 \normE ~ \expect[ (2\| \grad(X) \|+ \theta \sqrt{\normE} )^2 ]  , \label{eq:tmp2672}
 \end{align}
 where for the last inequality we used Assumption \eqref{eq:Assumption THETA}.
 Expanding the last term yields
 \begin{align}
  \expect[  (2\| & \grad(X) \|+  \theta \sqrt{\normE} )^2 ] \nonumber\\
  &= 4\normE +4 \theta \expect[\| \grad(X) \|] \sqrt{\normE} +\theta^2\normE \nonumber\\
  &\leq 4\normE +4 \theta \normE +\theta^2\normE \nonumber\\
  &= (2+\theta)^2 \normE . \label{eq:tmp439823731}
 \end{align}
 Here, we used the relation $\expect[\| \grad(X) \|]^{2} \leq \normE$, which holds true by Jensen's inequality.
 Combining \eqref{eq:tmp2672} and \eqref{eq:tmp439823731} and taking the supremum over $\|u\|\leq 1$ yields
 $$
  \| \expect[ A^2] \| \leq \theta^2 (2+\theta)^2 \normE^2 ,
 $$
 which gives (i). To show (ii), we first notice that Assumptions \eqref{eq:Assumption BETA} and \eqref{eq:Assumption THETA} yield
 \begin{equation}\label{eq:tmp28444}
  \|\gradapprox(X)\|^2 \leq (\| \grad(X)\| + \|\gradapprox(X)-\grad(X)\|)^2 \leq (\beta+\theta)^2\normE,
 \end{equation}
 almost surely.
 Then, for any $u\in\mathbb{R}^{\dimension}$ such that $\|u\|\leq 1$, we have
 \begin{align}
  u^T\expect[B^2]u 
  &= \expect[ u^T (\gradapprox(X)\gradapprox(X)^T)^2 u] \nonumber\\
  &= \expect[ \|\gradapprox(X)\|^2 (\gradapprox(X)^Tu)^2] \nonumber\\
  &\leq (\beta+\theta)^2\normE \expect[(\gradapprox(X)^Tu)^2].
  \label{eq:tmp3257}
 \end{align}
 Using similar arguments, the last term in the above relation satisfies
 \begin{align}
  \expect[(\gradapprox&(X)^Tu)^2] 
  =\expect[  (\gradapprox(X)^T u ) ^2  -  (\grad(X)^T u ) ^2 ] + \expect[(\grad(X)^T u )^2 ]\nonumber\\
  &=\expect[  ( (\gradapprox(X)+\grad(X))^T u ) (\gradapprox(X)-\grad(X)^T u ) ] + \expect[(\grad(X)^T u )^2 ] \nonumber\\
  &\leq\expect[  \| \gradapprox(X)+\grad(X))\| ~ \| \gradapprox(X)-\grad(X)\| ] + \normE \nonumber\\
  &\leq\expect[ (2\| \grad(X) \|+\|\gradapprox(X)-\grad(X)\| )  \| \gradapprox(X)-\grad(X)\| ] + \normE \nonumber\\
  &\leq 2\expect[ \| \grad(X) \|]\theta \sqrt{\normE}  +\theta^2\normE+ \normE \nonumber\\
  &\leq (2\theta+\theta^2+1)^2\normE = (1+\theta)^2\normE. \label{eq:tmp76544} 
 \end{align}
 Combining \eqref{eq:tmp3257} with \eqref{eq:tmp76544} and taking the supremum over $\|u\|\leq 1$ yields 
 $$
  \|\expect[B^2]\|\leq (\beta+\theta)^2(1+\theta)^2\normE^2,
 $$
 which establishes (ii). 
 This proves \eqref{eq:MFcharacteristics_VARIANCE}.
 \\
 
 Now we prove the second part of the property: the bound on the summand \eqref{eq:MFcharacteristics_BOUND}.
 Recall that $S_i$ is defined as $\frac{1}{m_1}(\expect[A]-A)$ if
 $1\leq i\leq m_1$ and as $\frac{1}{m_2} (\expect[B]-B)$ if $m_1\leq i\leq m$. 
 To obtain \eqref{eq:MFcharacteristics_BOUND}, it is then sufficient to show (iii) that $\|\expect[A]-A\|\leq 2\theta (2 \beta +  \theta )\normE$ almost surely and (iv) that $\|\expect[B]-B\| \leq 2(\beta+\theta)^2\normE$ almost surely.
 To show (iii), we first notice that
 \begin{align*}
  \|A\| 
  &= \| \grad(X)\grad(X)^T - \gradapprox(X)\gradapprox(X)^T \| \\
  &= \| \grad(X)(\grad(X)-\gradapprox(X))^T - (\gradapprox(X)-\grad(X))\gradapprox(X)^T \| \\
  &\leq \| \grad(X)\| \| \grad(X)-\gradapprox(X))\| + \| \gradapprox(X)-\grad(X) \| \| \gradapprox(X)\| \\
  &\leq \| \grad(X)-\gradapprox(X))\| ( 2 \| \grad(X)\| +  \|\grad(X)- \gradapprox(X)\| ) \\
  &\leq \theta\sqrt{\normE} (2 \beta \sqrt{\normE} +  \theta \sqrt{\normE}) \\
  &\leq \theta (2 \beta +  \theta )\normE . 
 \end{align*}
 Then, we can write
 $$
  \|\expect[A]-A\| \leq \|\expect[A]\|+\|A\| \leq 2\theta (2 \beta +  \theta )\normE,
 $$
 which gives (iii). Finally we have
 $$
  \|\expect[B]-B\|
  \leq \expect[\|\gradapprox(X)\|^2]+\|\gradapprox(X)\|^2
  \overset{\eqref{eq:tmp28444}}{\leq} 2(\beta+\theta)^2\normE ,
 $$
 which gives (iv) and therefore \eqref{eq:MFcharacteristics_BOUND}. This concludes the proof of Property \ref{prop:MFcharacteristics}.
\end{proof}

To prove our main result, Proposition \ref{prop:MainResult}, it remains to express the bounds of the estimator variance and the summands as a function of $m_{1}$ and to apply the matrix Bernstein theorem.
By Assumption \eqref{eq:m2geqm1}, we have $m_2\geq m_1 \frac{ (\theta+\beta)^2}{\theta(2\beta+\theta)}$ so that, using equation \eqref{eq:MFcharacteristics_BOUND} of Property \ref{prop:MFcharacteristics}, we have that
$$
  \Vert S_i \Vert \leq  \frac{2\theta(2\beta+\theta)}{m_1} \normE =:L,
$$
holds almost surely for all $1\leq i\leq m$.
By Assumption \eqref{eq:m2geqm1}, we also have $m_2\geq m_1 \frac{ (\theta+\beta)^2(1+\theta)^2}{\theta^2(2+\theta)^2}$ so that  equation \eqref{eq:MFcharacteristics_VARIANCE} yields 
$$
  \| \expect[ (S_1+\hdots+S_m)^2] \|
  \leq \frac{2\theta^{2} (2+\theta)^{2}}{m_1} \normE^{2} =: v .
$$
Applying the matrix Bernstein theorem (Theorem \ref{th:bernsteinNONintrinsic}) gives
\begin{align*}
 \expect[\Vert\matrixH-\HMF\Vert]
 &=\expect [\Vert S_1+\hdots+S_m \Vert]  \\
 &\leq  \sqrt{2 v \log(2\dimension) } + \frac{1}{3} L \log(2d) \\
 &= \Big(  \frac{2\theta(2+\theta)\sqrt{\log(2\dimension) }}{\sqrt{m_1}} +  \frac{2\theta(2\beta+\theta) \log(2\dimension) }{3 m_1} \Big) \normE \\
 &\overset{\eqref{eq:m1geq MEAN}}{\leq} (\varepsilon + \varepsilon^2 ) \frac{\normE}{\intdim} = (\varepsilon + \varepsilon^2 ) \|\matrixH\|,
\end{align*}
where, for the last equality, we used the definition of $\intdim= \trace(H)/\|\matrixH\|$ and the fact that $\trace(H)=\trace\expect(\grad(X)\grad(X)^T)=\normE$. 
This proves equation \eqref{eq:MainResult MEAN}.

To show equation \eqref{eq:MainResult PROBA}, we apply the high-probability bound of Theorem \ref{th:bernsteinNONintrinsic} with $t=\varepsilon \|\matrixH\|$. 
We obtain
\begin{align*}
  \proba\big\{\Vert\matrixH-\HMF\Vert \geq \varepsilon\normH \big\} 
  &\leq 2\dimension \exp\left(\frac{-\varepsilon^2 \|\matrixH\|^2/2}{v+L\varepsilon \|\matrixH\|/3}\right) \\
  &= 2\dimension \exp\left(\frac{-\varepsilon^2 m_1}{ 4\theta^2(2+\theta)^2 \intdim^{2}+4/3\theta(2\beta+\theta) \intdim \varepsilon }\right) 
  \overset{\eqref{eq:m1geq PROBA}}{\leq} \eta,
\end{align*}
which is equation \eqref{eq:MainResult PROBA}. This concludes the proof of Proposition \ref{prop:MainResult}.

\subsection{Proof of Proposition \ref{prop:SingleFidelityResult}}\label{sec:SingleFidelityResult}

The proof of Proposition \ref{prop:SingleFidelityResult} relies on the concentration inequality known as the intrinsic dimension matrix Bernstein theorem.
It corresponds to Property~7.3.1 and Corollary~7.3.2 from \cite{tropp_intro_MAL-048}, restricted to the case of real symmetric matrices.

\begin{theorem}[Matrix Bernstein: intrinsic dimension, real symmetric case]
\label{th:bernstein}
  Let $S_1,\hdots,S_m$ be $m$ zero-mean random symmetric matrices in $\mathbb{R}^{\dimension\times\dimension}$. 
  Assume the existence of $L<\infty$ such that 
  \begin{align*}
    \Vert S_i \Vert \leq L ,
  \end{align*}
  almost surely for all $1\leq i \leq m$. 
  Let $V\in\mathbb{R}^{\dimension\times\dimension}$ be a symmetric matrix such that
  $$
   \expect[ (S_1+\hdots+S_m)^2] \mleq V , 
  $$
  and let $\delta_{V} = \trace(V)/\|V\|$ and $v = \Vert V \Vert$.
  Then, we have
  \begin{align*}
    \expect [\Vert S_1+\hdots+S_m \Vert] & \leq C_{0} \left(\sqrt{v\ln(1+2\delta_{V})} + L\ln(1+2\delta_{V})\right),
  \end{align*}
  where $C_{0}$ is an absolute (numerical) constant. 
  Furthermore, for any $t\geq \sqrt{v}+L/3$, we have
  \begin{align*}
    \proba\{\Vert S_1+\hdots+S_m \Vert \geq t\} \leq 8 \delta_{V}\exp\left(\frac{-t^{2}/2}{v+Lt/3}\right).
  \end{align*}
\end{theorem}
\begin{proof}
  Using the definition of $V$ and the symmetry of the matrices $S_{i}$, we have:
  \begin{align}
    &V \mgeq \expect[ (S_1+\hdots+S_m)^2]= \expect[ (S_1+\hdots+S_m)(S_1+\hdots+S_m)^{T}]\\
    &V \mgeq \expect[ (S_1+\hdots+S_m)^2]= \expect[ (S_1+\hdots+S_m)^{T}(S_1+\hdots+S_m)].
  \end{align}
  Defining the matrix $M = \begin{pmatrix}V& 0\\ 0& V \end{pmatrix}$, we have $\delta_{M} = 2\delta_{V}$.
  A direct application of Property~7.3.1 and Corollary~7.3.2 from \cite{tropp_intro_MAL-048} yields the theorem.
\end{proof}

In order to apply the intrinsic dimension matrix Bernstein theorem, we express the difference between the AS matrix $\matrixH$ and the SF estimator $\HSF$ as the sum of independent matrices.
We write $\matrixH-\HSF = S_1 + \hdots + S_{m_1}$ where 
$$
 S_i = \frac{1}{m_1}\big( \matrixH - \grad(X_i)\grad(X_i)^T \big),
$$
for all $1\leq i \leq m_1$. 
Since $\matrixH=\expect[\grad(X)\grad(X)]$, we have
\begin{align*}
 \|S_i\| 
 &\leq \frac{\| \matrixH \| + \| \grad(X_i)\grad(X_i)^T \|}{m_1} 
 \leq \frac{\expect[\|\grad(X)\|^2] + \| \grad(X_i)\|^2}{m_1}     \\
 &\overset{\eqref{eq:Assumption BETA SF}}{\leq} \frac{ \normE + \beta^2\normE }{m_1}
 = \frac{ 1 + \beta^2}{m_1}\normE  =: L
\end{align*}
almost surely for all $1\leq i \leq m_1$. 
By independence of the summands $S_i$ and given $\expect[S_i]=0$, we have
\begin{align*}
 \expect[(S_1+& \hdots+S_{m_1})^2]
 = \expect[S_1^2] + \hdots + \expect[S_{m_1}^2] \\
 &= \frac{1}{m_1} \expect[(\grad(X)\grad(X)^T - \expect[\grad(X)\grad(X)^T]  )^2] \\
 &\mleq \frac{1}{m_1} \expect[(\grad(X)\grad(X)^T)^2] = \frac{1}{m_1} \expect[\|\grad(X)\|^2 \, \grad(X)\grad(X)^T ] \\
 &\overset{\eqref{eq:Assumption BETA SF}}{\mleq} \frac{\beta^2 \normE}{m_1} \expect[  \grad(X)\grad(X)^T ] 
 = \frac{\beta^2 \normE}{m_1} H =: V
\end{align*}
With the above definition of $V$, we have 
\begin{align*}
 & v := \|V\| = \frac{\beta^2 \normE \|\matrixH\|}{m_1}    \\
 & \delta := \frac{\trace(V)}{\|V\|} = \frac{\trace(H)}{\|\matrixH\|} = \intdim.
\end{align*}
Applying the expectation bound of Theorem \ref{th:bernstein} gives
\begin{align*}
 \expect[\Vert & \matrixH-\HSF \Vert] 
 = \expect [\Vert S_1+\hdots+S_m \Vert]  
  \leq C_{0} \left(\sqrt{v\ln(1+2\delta)} + L\ln(1+2\delta)\right) \\
 &\leq C_{0} \left(\sqrt{\frac{\beta^2 \normE \|\matrixH\|}{m_1}\ln(1+2\intdim)} + \frac{ 1 + \beta^2}{m_1}\normE\ln(1+2\intdim)\right) \\
 &\overset{\eqref{eq:m1geq SF}}{\leq} C_{0} \left( \sqrt{\frac{\varepsilon^{-2}\|H\|^2}{C}} + \frac{\varepsilon^{-2}\|H\|}{C} \right)
 \leq (\varepsilon+\varepsilon^2)\|\matrixH\| ,
\end{align*}
where the last inequality is obtained by defining $C=\max\{C_0,C_0^2\}$. 
This yields equation \eqref{eq:SingleFidelityResult MEAN}.

Finally, letting $t=\varepsilon\|\matrixH\|$, the high-probability bound of Theorem \ref{th:bernstein} ensures that 
\begin{align*}
 \proba\{\Vert \matrixH-\HSF \Vert \geq \varepsilon\|\matrixH\| \} 
 &\leq 8 \intdim \exp\left(\frac{-\varepsilon^2\|\matrixH\|^2/2}{v+L\varepsilon\|\matrixH\|/3}\right) \\
 &= 8 \intdim \exp\left(\frac{-\varepsilon^2\|\matrixH\|^2/2}{\frac{\beta^2 \normE \|\matrixH\|}{m_1}+\frac{ 1 + \beta^2}{m_1}\normE \varepsilon\|\matrixH\|/3}\right) \\
 &= 8 \intdim \exp\left(\frac{-m_1\varepsilon^2/2}{\beta^2 \intdim+(1 + \beta^2)\intdim \varepsilon/3}\right)
 \overset{\eqref{eq:m1geq SF PROBA}}{\leq} \eta,
\end{align*}
which is equation \eqref{eq:SingleFidelityResult PROBA}. 
This concludes the proof of Proposition \ref{prop:SingleFidelityResult}.


\bibliographystyle{siamplain}
\bibliography{biblio}
\end{document}